\documentclass[a4paper]{amsart}
\usepackage[T1]{fontenc}
\usepackage[utf8]{inputenc}
\usepackage{amsfonts,amssymb}
\usepackage{mathrsfs}
\usepackage{mathtools}			
	\mathtoolsset{centercolon}	
\usepackage{xcolor}			
\usepackage{comment}
\usepackage{graphicx}
\usepackage{import}			
\usepackage{calc}				
\usepackage{hyperref}
\usepackage{bookmark}

\theoremstyle{plain}
\newtheorem{thm}{Theorem}[section]
\newtheorem{cor}[thm]{Corollary}
\newtheorem{lem}[thm]{Lemma}
\newtheorem{prop}[thm]{Proposition}

\theoremstyle{definition}
\newtheorem{defi}[thm]{Definition}

\theoremstyle{remark}
\newtheorem{rem}[thm]{Remark}
\newtheorem{notation}[thm]{Notation}

\numberwithin{equation}{section}

\newenvironment{proofof}[2]{\begin{proof}[Proof of #1~\ref{#2}.]}{\end{proof}}
\newcommand{\beq}{\begin{equation}}
\newcommand{\eeq}{\end{equation}}

\newcommand{\numberset}{\mathbb}
\newcommand{\CC}{\numberset{C}}
\newcommand{\DD}{\numberset{D}}
\newcommand{\FF}{\numberset{F}}
\newcommand{\HH}{\numberset{H}}
\newcommand{\KK}{\numberset{K}}
\newcommand{\NN}{\numberset{N}}
\newcommand{\PP}{\numberset{P}}

\newcommand{\RR}{\numberset{R}}
\newcommand{\TT}{\numberset{T}}
\newcommand{\ZZ}{\numberset{Z}}

\DeclareMathOperator{\SL}{SL}
\DeclareMathOperator{\GL}{GL}
\DeclareMathOperator{\PSL}{PSL}
\DeclareMathOperator{\Aff}{Aff}

\DeclareMathOperator{\height}{height}
\DeclareMathOperator{\sys}{sys_{s.c.}}		
\DeclareMathOperator{\Id}{Id}
\DeclareMathOperator{\Realpart}{Re}
\DeclareMathOperator{\Imaginarypart}{Im}
\DeclareMathOperator{\Area}{Area}

\renewcommand{\epsilon}{\varepsilon}
\renewcommand{\phi}{\varphi}

\newcommand{\cA}{\mathscr{A}}

\newcommand{\cD}{\mathcal{D}}

\newcommand{\cK}{\mathcal{K}}
\newcommand{\cL}{\mathcal{L}}

\newcommand{\cU}{\mathcal{U}}

\newcommand{\cW}{\mathcal{W}}

\newcommand{\area}[2]{{\Area\bigl({#2},{#1}\bigr)}}
\newcommand{\areavert}[1]{{\Area\bigl({#1}\bigr)}}
\newcommand{\areaW}[2]{{\Area\bigl({#2}\bigr)}}
\newcommand{\ab}[1]{{\alpha_{#1}}} 			
\newcommand{\bb}[1]{{\beta_{#1}}}			
\newcommand{\cc}[1]{{c_{#1}}} 			
\newcommand{\g}{G} 					
\newcommand{\f}{F} 					
\newcommand{\peq}[1]{= {#1 }_\partial} 		

\newcommand{\Arc}{\mathcal{A}}
\newcommand{\W}[2]{{W^{(#1)}_{#2}}}
\newcommand{\vv}[2]{{v_{#2}^{(#1)}}}

\renewcommand{\emptyset}{\varnothing}

\DeclarePairedDelimiter{\abs}{\lvert}{\rvert}		

\newcommand{\reflection}{\begin{pmatrix} 0 & 1 \\ 1 & 0 \end{pmatrix}}

\newcommand{\slduer}{\SL(2, \RR)}
\newcommand{\pslduer}{\PSL(2,\RR)}
\renewcommand{\Re}{\Realpart}
\renewcommand{\Im}{\Imaginarypart}
\newcommand{\re}{\Re}
\newcommand{\im}{\Im}

\begin{document}

\title[The Lagrange spectrum of a Veech surface has a Hall ray]{The Lagrange spectrum of a Veech surface \\ has  a Hall ray}

\author[M. Artigiani]{Mauro Artigiani}
\address{School of Mathematics \\ University of Bristol \\ University Walk \\ Bristol \\ BS8~1TW \\ United Kingdom}
\email{mauro.artigiani@bristol.ac.uk}

\author[L. Marchese]{Luca Marchese}
\address{LAGA \\ Universit\'e Paris 13 \\  Avenue Jean-Baptiste Cl\'ement \\ 93430 Villetaneuse \\ France}
\email{marchese@math.univ-paris13.fr}

\author[C. Ulcigrai]{Corinna Ulcigrai}
\address{School of Mathematics \\ University of Bristol \\ University Walk \\ Bristol \\ BS8~1TW \\ United Kingdom}
\email{corinna.ulcigrai@bristol.ac.uk}





\begin{abstract}
We study  Lagrange spectra of Veech translation surfaces, which are a generalization of the classical Lagrange spectrum.
We show that any such Lagrange spectrum contains a Hall ray.
As a main tool, we use the boundary expansion developed by Bowen and Series to code geodesics in the corresponding Teichm\"uller disk and prove a formula which allows to express large values in the Lagrange spectrum as sums of Cantor sets.
\end{abstract}

\maketitle


\section{Introduction}
The Lagrange spectrum $\cL$ is a classical and much studied subset of the extended real line, which can be described either geometrically or number theoretically. 
In connection with Diophantine approximation, it is the set
\[
	\cL := \left\{ L(\alpha):= \limsup_{q,p\to\infty} \frac{1}{q|q\alpha -p|},  \alpha \in\RR \right\} \subset \overline{\mathbb{R}}:=\mathbb{R} \cup \{+\infty\}.
\]
In other words, $L \in \cL$ if and only if there exists $\alpha\in\RR$ such that, for any $c>L$, we have $|\alpha - p/q| > 1/cq^2$ for all $p$ and $q$ big enough and, moreover, $L$ is minimal with respect to this property. 
One can show that $\cL$ can also be described as a penetration spectrum for the geodesic flow on the (unit tangent bundle of the) modular surface $X=\HH/ \SL(2,\ZZ)$ in the following way.
If $(\gamma_t)_{t\in \RR}$ is any hyperbolic geodesic on $X$ which has $\alpha \in \RR$ as forward endpoint, the  value $L(\alpha)$ is related to the geometric quantity
\begin{equation}\label{eq:spectrumheight}
	\limsup_{t\to+\infty} \height (\gamma_t)
\end{equation}
where $\height (\cdot)$ denotes the hyperbolic height function. This quantity gives the asymptotic depth of penetration of the geodesic $\gamma_t$ into the cusp of the modular surface.
 
The structure of $\cL$ has been studied for more than a century, from the works of Markoff (1879). Hurwitz (1891) and Hall (1947) up to the recent results by Moreira~\cite{Mo}. We refer the interested reader to the book~\cite{CF} by Cusick and Flahive.     
Moreover, several generalizations of the classical Lagrange spectrum have been studied by many authors, in particular in the context of Fuchsian groups and, more in general, negatively curved manifolds see~\cite{Ferenczi,HaasSeries,HersonskyPaulin, Maucourant, paulin2, schmidtsheingorn, Series:Markoff, Vulakh}. For a very brief survey of these generalizations, we refer to the introduction of~\cite{HMU}. 
 
In particular, a generalization of Lagrange spectrum was recently defined in~\cite{HMU} in the context of \emph{translation surfaces}, which are surfaces obtained by glueing a finite set of polygons in the plane, identifying pairs of isometric parallel sides by translations.
The simplest example of a translation surface is a flat torus, obtained by identifying opposite parallel sides of a square.
More in general, if we start from a regular polygon  with $2n$ sides, with $n\geq 4$, we obtain a translation surface of higher genus, for example of genus $2$ for the regular octagon. 
Translation surfaces carry a flat Euclidean metric apart from finitely many conical singular points and can equivalently be defined as Riemann surfaces with Abelian differentials (see for example the survey by Masur~\cite{Masur} or the lecture notes by Viana or Yoccoz~\cite{Viana,Yoccoz}).
These surfaces have been object of a great deal of research in the past thirty years, in connection with the study of interval exchange transformations (IETs), billiards in rational polygons and the Teichm\"uller geodesic flow (see for example the surveys~\cite{Masur,Viana,Yoccoz,Zorich}).

In this paper, we study Lagrange spectra of \emph{Veech translation surfaces}.
These are special translation surfaces which have many symmetries (the definition is given in \S~\ref{sec:background}).
For instance, the surfaces obtained by glueing regular polygons with $2n$ sides are all examples of Veech translation surfaces.
One of the characterizing properties of Veech translation surfaces is that the moduli space of their affine deformations is the unit tangent bundle of a hyperbolic surface, called \emph{Teichm\"uller curve} (see \S~\ref{sec:background}). 
 
A \emph{saddle connection $\gamma$} on a translation surface $S$ is a geodesic segment for the flat metric of $S$ starting and ending in a conical singularity and not containing any other conical singularity in its interior. 
To each saddle connection $\gamma$ we can associate a vector in $\RR^2$, called \emph{displacement vector} (or holonomy), which can be obtained developing $\gamma$ to $\RR^2$ and then taking the difference between the final and initial point of the flat geodesic in $\RR^2$. 
In the following, we will often abuse the notation and say ``let $v$ be a saddle connection''  when $v \in \RR^2$ is the displacement vector of a saddle connection. 

Let $v$ be (the displacement vector of) a saddle connection on $S$ and let $\theta $ be a fixed direction.
Then $\areavert{v}$ is by definition the area of a rectangle with horizontal and vertical sides which has $v$ as a diagonal.
If we identify  $\RR^2$ with the complex plane $\CC$ and denote by $\Re(v)$ and $\Im(v)$ the real and imaginary part of $v$, we have that
$\areavert{v}=|\Re(v)|\cdot|\Im(v)|$.
More in general, we will denote by $\area{\theta}{v}$  the area of a rectangle which has sides parallel to $\theta$ and its perpendicular direction $\theta_\perp$ and $v$ as a diagonal.
That is 
\[
	\area{\theta}{v}=|\Re_\theta(v)|\cdot|\Im_\theta(v)|, 
\]
where  $\Re_\theta(v): = \Re (e^{i\theta} v)$ and $\Im_\theta(v): = \Re (e^{i\theta} v)$ are respectively the real and imaginary part of the rotated vector $e^{i\theta} v \in \CC$.

The definition of Lagrange spectrum for a translation surface given in~\cite{HMU} reduces, in the case when $S$ is a Veech surface of total area one, to $\cL(S) = \{L_S(\theta), 0\leq \theta < 2\pi \}$, where
\beq\label{eq:defLalpha}
L_S(\theta):=\limsup_{|\Im_\theta(v)|\to\infty} \frac{1}{\area{\theta}{v}} = \frac{1}{\liminf_{|\Im_\theta(v)|\to\infty} \area{\theta}{v}}.
\eeq
One can show (see~\cite{HMU}) that this gives a generalization of the classical Lagrange spectrum, since it reduces to the classical definition when $S$ is the flat torus $\TT^2 / \ZZ^2$. 
The same quantity, in analogy with the classical case, has also a definition of Diophantine nature in terms of interval exchange transformations (which can be found in~\cite{HMU}) and also a hyperbolic equivalent definition, analogous to~\eqref{eq:spectrumheight}, which we will now explain.
As we said above, if $S$ is a Veech surface, the space of its affine deformations is the unit tangent bundle to a non-compact finite-volume hyperbolic surface $X$, called \emph{Teichm\"uller curve}, as explained in \S~\ref{sec:background}.
The \emph{(saddle connection) systole} function $S' \mapsto\sys(S')$ on $T^1 X$ is defined as the \emph{length of the shortest saddle connection} of the translation surface $S'$, that is
\[
	\sys(S'):=\min\{|v|, v \text{ displacement vector of a saddle connection on } S' \}.
\]
Let $(g_t)_t$  be the geodesic flow on $T^1 X$.
Then a dynamical estimate of the asymptotic maximal excursion of the positive orbit $(g_t(S'))_{t\geq 0}$ into the cusps of $T^1 X$ is given by
\beq\label{eq:definitions}
	s(S'):=\liminf_{t\to\infty} \sys(g_t(S')).
\eeq
This quantity is related to $L_S (\theta)$ by the following formula, due to Vorobets (see~\cite{Vor:pla} and also \S1.3 of~\cite{HMU}):
\[
	L_S(\theta) =  \frac{2}{s^2(r_\theta S)},
\]
where $r_\theta S \in T^1 X$ is the surface obtained by rotating $S$ by an angle $\theta$ in the anticlockwise direction.

A fundamental result on the classical Lagrange spectrum was proven in~\cite{Hall} by Hall in 1947, who showed that $\cL$ contains a positive half-line, i.e.\ there exists $r$ such that the interval $[r,+\infty] \subset \cL$. Any interval with this property is now called a \emph{Hall's ray}. 
The value of $r$ was later improved by Cusick and others~\cite{CF} until finally Freiman in 1973 computed precisely the smallest $r$ with this property, see~\cite{Freiman}. 
Our main result generalizes the classical result by Hall to any Veech surface.
We prove the following:

\begin{thm}[Existence of Hall ray]\label{thm:HallVeech}
Let $S$ be a Veech translation surface and let $\cL(S)$ be its Lagrange spectrum.
Then $\cL(S)$ contains a \emph{Hall ray}, that is there exists $r= r(S)>0$ such that 
\[
	[r(S),+\infty] \subset \cL(S).
\]
\end{thm}
Let us now comment on how our Theorem~\ref{thm:HallVeech} relates to existing results in the literature. 
For a special case of Veech surfaces, i.e.\ \emph{square-tiled} surfaces, which are translation surfaces that are covers of the flat torus branched over a single point (also known as \emph{origamis} or arithmetic Veech surfaces), the existence of a Hall ray was proved in~\cite{HMU}.
However, the proof was very specific to square-tiled surfaces.
The result in~\cite{HMU} guarantees the existence of a Hall ray for the Lagrange spectrum of any closed $\SL(2,\RR)$-invariant locus of translation surfaces containing a square-tiled surface, see Corollary~1.7 in~\cite{HMU}.
In the same way, one deduces from Theorem~\ref{thm:HallVeech} the stronger result that the Lagrange spectrum of every closed invariant locus containing a \emph{Veech surface} has a Hall ray.
This leaves open the natural question whether the existence of a Hall ray can be proved also for invariant loci which do not contain Veech surfaces (such loci are known to exist by the work of McMullen~\cite{McMg2}).

As we remarked above (see the discussion around~\eqref{eq:definitions}), the Lagrange spectra that we consider are related to penetration spectra in the context of hyperbolic surfaces.
More precisely, one can define Lagrange spectra as the penetration spectra of geodesics into the cusps of a non-compact (finite volume) hyperbolic surface $X$.
This definition reduces to $\cL(S)$ when $X$  is the Teichm\"uller curve  of a Veech surface $S$ and the penetration is measured by the function $1/\sys$. 
The question of existence of a Hall ray in the general context of non-compact (finite volume) hyperbolic surfaces is currently open and thus our result constitutes a first step in this direction.
On the other hand, Schmidt and Sheingorn proved in~\cite{schmidtsheingorn} the existence of a Hall ray for the \emph{Markoff spectrum} of penetration of geodesics into the cusps of $X$, which is a close relative of the Lagrange spectrum itself.
One can show that this result implies the existence of a Hall ray for the set of values
\beq\label{eq:Markoff}
\left\{ \sup_{t>0}\frac{2}{\sys(g_t \rho_\theta S)^2}, \quad \text{ for }0\leq  \theta <2\pi \right\}.
\eeq
We stress that such result does not imply that there is a Hall ray for the Lagrange spectrum $\cL(S)$.
In general, indeed, it is much easier to construct values in the Markoff spectrum than in the Lagrange spectrum, essentially because in order to show that a certain value is in the former it is enough to construct a geodesic along which the supremum in~\eqref{eq:Markoff} is achieved, while for the latter one has to construct a sequence of times which tends to the given value.  

Two different ways of further generalizing Lagrange spectra of hyperbolic surfaces are to either consider variable curvature or higher dimension.
In both cases there are results in the direction of the existence of Hall rays.
However, as we said above, this problem is still open in dimension $2$ even for the case of constant negative curvature.
In a very recent work~\cite{MR}, G.~Moreira and S.~A.~Roma\~{n}a considered the case of surfaces with \emph{variable} negative curvature.
They proved that the corresponding spectra have non-empty interior, i.e.\ they contain an \emph{interval} (not necessarily a Hall \emph{ray}) for a generic small $C^k$-perturbation of a hyperbolic metric on the punctured surface, where $k\geq 2$.  
In the context of hyperbolic manifolds of negative curvature in dimension $n\geq 3$, Paulin and Parkkonen in~\cite{paulin2} showed the existence of Hall rays for the associated Lagrange spectra.
Moreover, they also give a quantitative \emph{universal} bound on the beginning of the Hall ray.
Unfortunately our methods do not allow us to give quantitative estimates on $r(S)$ in Theorem~\ref{thm:HallVeech}.
It is clear though that such estimates cannot be independent of the topological complexity of the surface, since as shown in~\cite{HMU} the minimum of $\cL(S)$ grows with the genus and the number of singularities (see Lemma~1.3 in~\cite{HMU}). 


Our proof follows the scheme of Hall.
A crucial point for his proof was the formula which allows to compute values of the Lagrange spectrum in terms of the continued fraction entries, that is 
\beq\label{eq:classical_formula}
	L_{\TT^2}(\theta)=\limsup_{n\to\infty}{[a_{n-1}, a_{n-2}, \dots, a_1] + a_n + [a_{n+1}, a_{n+2}, \dots]},
\eeq
where $(a_n)_n$ are the continued fraction entries of the tangent of the angle $\theta$.
One of our main results is that a similar formula holds for high values in the spectrum. In our formula, the continued fraction expansion entries are replaced by cutting sequences of hyperbolic geodesics and in particular  by the boundary expansions invented and studied by Bowen and Series (see~\cite{BowenSeries}).
We then exploit the interplay between the hyperbolic and flat worlds to study areas of a carefully chosen set of saddle connections which realize high values of the spectrum and can be calculated in terms of boundary expansions.

\subsection*{Structure of the paper}
The rest of the paper is arranged as follows. 
In \S~\ref{sec:background}, we define translation surfaces and Veech surfaces.
Moreover, we recall the background material we need about boundary expansions of hyperbolic geodesics.
In \S~\ref{sec:Lagrangeboundary} we show how one can use boundary expansions and the interplay between hyperbolic and flat worlds to compute large values in the Lagrange spectrum $\cL(S)$. In \S~\ref{sec:acceleration} we define an acceleration of  the boundary expansion which still allows to compute values in the spectrum and prove some bounds on values in terms of the number of steps of the acceleration.
Finally, in \S~\ref{sec:Hall}, we show the existence of a Hall's ray, proving our main Theorem~\ref{thm:HallVeech}. 
The main tool is the formula appearing in Theorem~\ref{thmrenormalizedformula}, which generalizes the classical~\eqref{eq:classical_formula} and enables us to express values in the Lagrange spectrum $\cL(S)$ as a sum of Cantor sets and then show the existence of a Hall's ray in a similar way than in  the classical proof of Hall.
Section~\ref{sec:renormalized_formula} is devoted to the proof of Theorem~\ref{thmrenormalizedformula}, while in \S~\ref{sec:Cantorproof} we prove some technical estimates on the Cantor sets which shows that their sum contains an interval.

\section{Background}\label{sec:background}
\subsection{Basic notation}
Let $\NN=\{0,1, \dots \}$ denote the natural numbers.
Let $\CC$ be the complex plane, which we consider identified with $\RR^2$. 
Let $\HH = \{ z \in \CC , \Im z>0 \} $ denote the upper half plane and let $\DD = \{ z \in \CC , |z|<1 \} $ denote the unit disk.
We will use $\HH$ and $\DD$ interchangeably, by using the identification $\mathscr{C} \colon \HH \to \DD$ given by 
\[
	\mathscr{C}(z)= \frac{z-i}{z+i}, \qquad z \in \HH.
\]
Let $\SL(2, \RR)$ be the set of $2\times 2$ matrices with real entries and determinant one. Given $A \in \SL(2, \RR)$ and $v \in \RR^2 $ we denote by $A v$ the linear action of $A$ on $v$.  The group $\SL(2, \RR)$ also acts on $\HH$ by \emph{M\"obius transformations} (or \emph{homographies}).
Given $\g \in \SL(2, \RR)$ we will denote by $\g*z$ the action of $\g$ on $z \in \HH$ given by
\[
z \mapsto \g*z = \frac{a z+b}{c z+d}, \qquad \text{if} \quad \g= \begin{pmatrix}a&b\\ c&d \end{pmatrix}.
\] 
This action of $\SL(2, \RR)$ on $\HH$ induces an action on the unit tangent bundle $T^1 \HH$, by mapping a unit tangent vector at $z$ to its image under the derivative of $\g$ in $z$, which is a unit tangent vector at $\g*z$.
This action is transitive but not faithful and its kernel is exactly $\{ \pm \Id \}$, where $\Id$ is the identity matrix.
Thus, it induces an isomorphism between  $T^1 \HH$ and $\PSL(2, \RR) =\SL(2, \RR) / \{ \pm \Id \}$. Throughout the paper, we will often write $A \in \PSL(2, \RR)$ and denote by $A \in \SL(2,\RR)$ the equivalence class of the matrix $A$ in $\PSL(2, \RR)$. Equality between matrices in $\PSL(2,\RR)$ must be intented as equality as equivalence classes. 
Finally, let $\RR \PP^1$ denote the set of \emph{unoriented directions} in $\RR^2$. We will adopt the convention of identifying a direction in $\RR \PP^1$ with the angle $\theta$ chosen so that $-\frac{\pi}{2} \leq \theta < \frac{\pi}{2}$ and $\theta$ is the angle formed by the direction with the vertical axis measured clockwise. When we write $\theta \in \RR\PP^1$ we hence assume that $-\frac{\pi}{2} \leq \theta < \frac{\pi}{2}$. 

\subsection{Translation surfaces and affine deformations}
A translation surface is a collection of polygons $P_j\subset\RR^2\cong\CC$ with identifications of pairs of parallel sides so that (1) sides are identified by maps which are restrictions of translations, (2) every side is identified to some other side and (3) when two sides are identified the outward pointing normals point in opposite directions. If $\sim$ denotes the equivalence relation coming from identification of sides then we define the surface $S=\bigcup P_j/\sim$. A translation surface inherits from the plane $\RR^2$ a Euclidean flat metric at all points apart from a finite set $\Sigma$ of singularities which is contained in the vertices of the polygons.
A saddle connection is a geodesic for the flat metric that connects two singularities, not necessarily distinct, and which does not contain any singularity in its interior.
A \emph{cylinder} is a subset of the surface which is the image of an isometrically immersed flat cylinder of the form $\RR/ w\ZZ \times h$, for some numbers $w$ and $h$ in $\RR_+$. 
Remark that any cylinder is foliated by  a collection of periodic flat geodesics, which for short will be called \emph{closed geodesics}.
One can show that any closed geodesic comes with a cylinder of parallel closed geodesics. 
 
The group $\GL^+(2, \RR)$ of two by two real matrices with positive determinant acts naturally on translation surfaces.
Given $\nu \in \GL^+(2, \RR)$ and a translation surface $S =\bigcup P_j/\sim$, denote by $\nu  P_j \subset \RR^2$ the image of $P_j \subset \RR^2$ under the linear map $\nu$.
Since $\nu$ takes pairs of parallel sides to pairs of parallel sides, it preserves the identifications between the sides of the polygons.
The surface obtained by glueing the corresponding sides of $\nu P_1, \dots, \nu P_n $ will be denoted by $\nu \cdot S$ and can be thought of as an affine deformation of $S$. 
We will consider in  particular the  action of the following $1$-parameter subgroups of deformations:
\[
	g_t =	\begin{pmatrix}
			e^{t/2}	 & 	0 \\
			0 	& 	e^{-t/2}
		\end{pmatrix}
	\qquad \text{and} \qquad
	r_\theta=\begin{pmatrix}
			\cos\theta 	&	 -\sin\theta \\
			\sin\theta 	&	 \cos\theta
		     \end{pmatrix}.
\] 
The first deformation, which stretches the horizontal direction and contracts the vertical, is called \emph{Teichm\"uller geodesic flow},   while the second corresponds simply to rotating the surface, and hence changing the vertical direction. 

Let $S$ and $S'$ be translation surfaces.
Consider a homeomorphism $\Psi$ from $S$ to $S'$ which takes the singular points $\Sigma$ of $S$ to the singular points $\Sigma'$ of $S'$ and is a diffeomorphism outside of $\Sigma$.
We can identify the derivative $D\Psi_p$ with an element of $\GL^+(2,\RR)$.
We say that $\Psi$ is an \emph{affine diffeomorphism} if the derivative $D\Psi_p$ does not depend on the point $p$. 
We say that $S$ and $S'$ are \emph{affinely equivalent} if there is an affine diffeomorphism $\Psi$ between them.  
We say that $S$ and $S'$ are \emph{translation equivalent} if they are affinely equivalent with $D\Psi=\Id$.
If $S$ is given by identifying sides of polygons $P_j$ and $S'$ is given by identifying sides of polygons $P'_k$  then a translation equivalence $\Upsilon$ from $S$ to $S'$ can be given by a \emph{``cutting and pasting''} map.
That is to say we can subdivide the polygons $P_j$ into smaller polygons and define a map $\Upsilon$ so that the restriction of $\Upsilon$ to each of these smaller polygons is a translation and the image of $\Upsilon$ is the collection of polygons $P'_k$.  
An affine diffeomorphism from $S$ to itself is an \emph{affine automorphism}. 
The collection of affine automorphisms is a group which we denote by $\Aff(S)$.
We can realize an affine automorphism of $S$ with derivative $\nu$ as a composition of an affine deformation from  $ S$ to $\nu\cdot S$ with a translation equivalence, or cutting and pasting map, from $ \nu\cdot S$ to $S$. 

\subsection{Veech translation surfaces and Teichm\"uller curves}\label{sec:Veechcurves}
The Veech homomorphism is the homomorphism $\Psi\mapsto D\Psi$  from $\Aff(S)$ to $\GL^+(2,\RR)$.
The image $V(S)$ of this homomorphism lies in $\SL(2,\RR)$ and is called the \emph{Veech group} of $S$.
Remark that the elements of the Veech group stabilize the surface $S$ under the action of $\SL(2,\RR)$.
Note that the term Veech group is used by some authors to refer to the image of the group of affine automorphisms in the projective group $\PSL(2,\RR)$. 

A translation surface $S$ is called a \emph{Veech surface} (or a \emph{lattice} surface) if  $V(S)$ is a lattice in $\SL(2, \RR)$.
The torus $\TT^2=\RR^2 / \ZZ^2$ is an example of a lattice surface whose Veech group is $\SL(2, \ZZ)$.
Veech more generally proved that all translation surfaces obtained from regular polygons are Veech surfaces (see~\cite{Ve:tei}). 
Veech surfaces satisfy the \emph{Veech dichotomy} (see~\cite{Ve:tei,Vor:pla}) which says that if we consider a direction $\theta$ then one of the following two possibilities holds: either there is a saddle connection in direction $\theta$ and the surface decomposes as a finite union of cylinders each of which is a union of a family of closed geodesics in direction $\theta$ or each trajectory in direction $\theta$ is dense and uniformly distributed.


Let $S$ be a Veech surface.
Since its Veech group $V(S)$ is a lattice in $\SL(2, \RR)$ then $X= \HH/  V(S)$ is a non-compact, finite volume hyperbolic surface, which we call the \emph{Teichm\"uller curve of $S$}.
We can moreover identify the unit tangent bundle $T^1 X$ of $X$ with the affine deformations modulo translation equivalence of the translation surface $S$.
Finally $T^1\DD$ can be identified with all affine deformations of $S$ as follows. 
We first use the identification induced by $\mathscr{C}$ between $T^1 \DD$ and $T^1 \HH$ and then the one between the latter and $\PSL(2, \RR)$ given by the action of M\"obius transformations.
In virtue of this identification, we will often call the space of $\SL(2, \RR)$ deformations of $S$ its \emph{Teichm\"uller disk} (for more details\footnote{Remark, however, that our conventions are slightly different from those of~\cite{SmillieUlcigrai}.}, we refer to~\cite{SmillieUlcigrai}).
We choose the convention that the center of the hyperbolic disk represents the surface $S$ and that the vertical direction on $S$ is represented by the unit tangent vector $i$. 
Let us remark that the above identification of $T^1 \DD$ with the Teichm\"uller disk induces an isomorphism between $\partial \DD$ and $\RR \PP^1$, i.e.\ the set of \emph{unoriented} directions. Given   $\xi$ in $\partial\DD$ we will denote by $\theta=\theta(\xi)$ the angle $-\frac{\pi}{2} \leq \theta < \frac{\pi}{2}$ in $\RR^2$ such that $\theta$ is the angle formed by the direction corresponding to $\xi$ with the vertical direction, measured clockwise.  Reciprocally, for any $\theta$ such that  $-\frac{\pi}{2} \leq \theta < \frac{\pi}{2}$ we will denote by $\xi=\xi(\theta)$ the corresponing point in $\partial\DD$. 

\begin{rem}\label{directions}
The above correspondence  is such that for any point $\xi$ in $\partial\DD$ the direction $\theta=\theta(\xi)$ on the translation surface $S$ is  such that  $-\frac{\pi}{2} \leq \theta < \frac{\pi}{2}$ and the lift $g_t^\theta$ of the geodesic flow to $T^1 \DD$ satisfies 
\[
	\lim_{t\to + \infty} g_t^\theta \cdot S=\xi, \qquad \text{where }  g_t^\theta = r_{\theta}^{-1} g_t r_{\theta}.
\]
\end{rem}
\noindent In particular, the geodesic $(g_t S)_{t\in \RR}$ converges to the point $e^{\frac{\pi}{2}i}=(0,1) \in \partial \DD$, while $(g_t^\theta S)_{t\in \RR}$ converges to the point $e^{( \frac{\pi}{2}+2\theta)i} \in \partial \DD$, which, as $\theta$ changes from $-\frac{\pi}{2}$ to $\frac{\pi}{2}$, rotates clockwise from $(0,-1)$ exactly once around $\partial \DD$.

\subsection{Cutting sequences and boundary expansions}\label{subsection:boundary}
For a special class of Fuchsian groups, Bowen and Series developed a geometric method of symbolic coding of points on $\partial \DD$, known as \emph{boundary expansions}, that allows to represent the action of a set of suitably chosen generators of the group as a subshift of finite type. 
Boundary expansions  can be thought of as a geometric generalization of the continued fraction expansion, which is related to the boundary expansion of the geodesic flow on the modular surface (see~\cite{Series:Modular} for this connection).  
We will now recall two  equivalent definitions of the simplest case of boundary expansions, either as cutting sequences of geodesics on $X=\HH/\Gamma$ or as itineraries of expanding maps on $\partial \mathbb{D}$.
For more details and a more general treatment we refer to the expository introduction to boundary expansions given by Series in~\cite{Series:TS}.

Let  $\Gamma \subset \PSL(2,\RR)$ be a Fuchsian group, namely a discrete groups of hyperbolic isometries.
Since hyperbolic isometries are given by M\"obius transformations, we will identify a hyperbolic isometry with the matrix $G \in \PSL(2, \RR)$ which give the transformation. Assume that $\Gamma$ be a co-finite, non cocompact and does not contain elliptic elements.
Equivalently, assume the quotient $X= \mathbb{H}/\Gamma$ is a smooth, non compact, hyperbolic surface with finite volume.
One can see that $\Gamma$  admits a fundamental domain which is an \emph{ideal polygon} $D$ in $\DD$, that is a hyperbolic polygon having finitely many vertices $\xi$ all lying on $\partial\DD$ (see for example Tukia~\cite{Tukia}).
We will denote by $e$ the sides of $D$, which are geodesic arcs with endpoints in $\partial\DD$.
Geodesic sides appear in pairs, i.e.\ for each $e$ there exists a side $\overline{e}$ and an element $G$ of $\Gamma$ such that the image $G( e)$ of $e $ by $G$ is $\overline{e}$. Let  $2d$ ($d\geq 2$) be the number of sides of of $D$. 
Let $\cA_0$  be a finite alphabet of cardinality $d$ and label the $2d$-sides ($d\geq 2$) of $D$ by letters in 
\[
	\cA = \cA_0 \cup \overline{\cA_0} = \{ \alpha \in \cA_0\} \cup \{ \overline{\alpha}, \alpha \in \cA_0\} 
\]
in the following way.
Assign to a side $e$ an internal label $\alpha$ and an external one $\overline{\alpha}$.
The side $\overline{e}$ paired with $e$ has $\overline{\alpha}$ as internal label and $\alpha$ as the external one.
We then  see that the pairing given by $G (e) =\overline{e}$ transports coherently the couple of labels of the side $e$ onto the couple of labels of the side $\overline{e}$. Let us denote by $e_\alpha$  the side of $D$ whose \emph{external} label is $\alpha$. 
A convenient set of generators for $\Gamma$ is given by the family of isometries $G_\alpha \in \PSL(2, \RR)$ for $\alpha\in\cA_0$, where $G_\alpha$ is the isometry which sends the side $e_{\overline{\alpha}}$ onto the side $e_{\alpha}$, and their inverses $G_{\overline{\alpha}}:= G_\alpha^{-1}$ for $\alpha\in\cA_0$, such that $G_\alpha^{-1}(e_{\alpha})=e_{\overline{\alpha}}$.  
Thus, $\cA$ can be thought as the set of labels of generators, see Figure~\ref{fig:fundomain}.
It is convenient to define an involution on $\cA$ which maps $\alpha \mapsto \overline{\alpha} $ and $\overline{\alpha}  \mapsto  \overline{\overline{\alpha}} = \alpha$. 

\begin{figure}
\centering
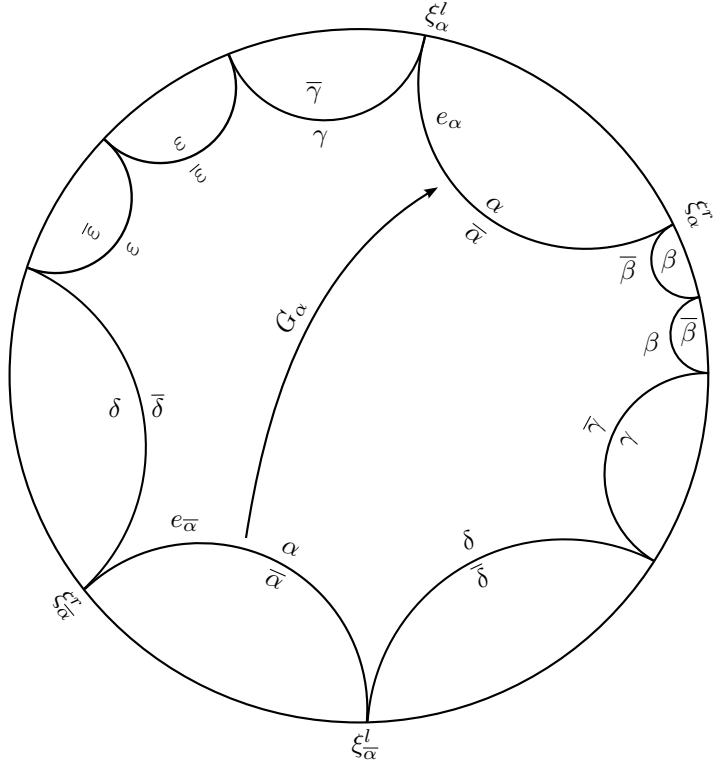
\caption{A hyperbolic fundamental domain, with sides labelling and the action of the generator $G_\alpha$.}
\label{fig:fundomain}
\end{figure}

Since $D$ is an ideal polygon, $\Gamma$ is a free group.
Hence every element of $\Gamma$ as a unique representation as a \emph{reduced word} in the generators, i.e.\ a word in which an element is never followed by its inverse.
We transport the internal and external labelling of the sides of $D$ to all its copies in the tessellation by ideal polygons given by all the images $G( D)$ of $D$ under  $G \in \Gamma$.
We label a side of a copy $G( D)$ of $D$ with the labels of the side $G^{-1}(e) \in \partial D$.
Remark that this is well defined since we have assigned an internal and an external label to each side of $D$, and this takes into account the fact that every side of a copy $G (D)$ belongs also to another adjacent copy $G'(D)$. Let  $\gamma$  be a hyperbolic geodesic ray, starting from the center $O$ of the disk and ending at a point $\xi \in \partial \DD$.
The \emph{cutting sequence} of $\gamma$ is the infinite reduced word obtained by concatenating the \emph{exterior} labels of the sides of the tessellation crossed by $\gamma$, in the order in which they are crossed.
In particular, if the cutting sequence of $\gamma$ is $\ab{0}, \ab{1}, \dots$, the $i^{th}$ crossing along $\gamma$ is from the region $G_{\ab{0}} \dots G_{\ab{i-1}} (D)$ to $G_{\ab{0}} \dots G_{\ab{i}}(D)$ and the sequence of sides crossed is
 \[
 	G_{\ab{0}} G_{\ab{1}} \cdots G_{\ab{n-1}} (e_{\ab{n}}), \qquad n \in \NN.
 \]
Remark that, since two distinct hyperbolic geodesics meet at most in one point, a word arising from a cutting sequence is reduced.
In other words, hyperbolic geodesics \emph{do not backtrack}. 
More in general, one can associate cutting sequences to any oriented piece of a hyperbolic geodesics and describe a cross-section of the geodesic flow in terms of the shift map on boundary expansions (see for example Series~\cite{Series:Modular}), 
but we will not need it in this paper.
  
Let us now explain how to recover cutting sequences of geodesic rays by itineraries of an expanding map on $\partial \DD$.
The action of each $\g \in \Gamma$ extends by continuity to an action on $\partial \DD$ which will be denoted by $\xi \mapsto \g(\xi)$.   
Let $\Arc [\alpha]$ be the shortest (closed) arc on $\partial\DD$ which is cut off by the edge $e_{\alpha}$ of $D$. 
Then it is easy to see from the geometry that the action $G_\alpha \colon \partial \DD \to \partial \DD$ associated to the generator $G_\alpha$ of $\Gamma$ sends the complement of $\Arc [\overline{\alpha}]$ to $\Arc [\alpha]$. Moreover, if for each $\alpha \in \cA$ we denote by $\xi_\alpha^l$ and $\xi_\alpha^r$ the endpoints of the side $e_\alpha$, with the  convention that the right follows the left moving in clockwise sense on $\partial \DD$, we have 
\beq \label{action_g_on_labels}
	G_\alpha(\xi_{\overline{\alpha}}^r)=\xi_\alpha^l
	\qquad \text{and} \qquad
	G_\alpha(\xi_{\overline{\alpha}}^l)=\xi_\alpha^r.
\eeq  
Let $\Arc = \bigcup_{\alpha} \stackrel{\circ}{ \Arc[\alpha]}\subseteq \partial \DD$, where $\stackrel{\circ}{ \Arc[\alpha] }$ denotes the arc $\Arc[\alpha]$ without endpoints. 
Define $\f \colon \Arc \to \partial\DD$ by 
\[
	\f(\xi) = G_{\alpha}^{-1}(\xi), \qquad \text{if } \xi \in \stackrel{\circ}{ \Arc[\alpha]}. 
\]
Let us call a point $\xi \in \partial D$ \emph{cuspidal} if it is an endpoint of the ideal tessellation with fundamental domain $D$, \emph{non-cuspidal} otherwise. One can see that $\xi$ is non-cuspidal point if and only if $\f^n(\xi)$ is defined for any $n \in \NN$.  
One can \emph{code} a trajectory $\{ \f^n(\xi), n \in \NN\}$ of a non-cuspidal point $\xi \in \partial \DD$ with its \emph{itinerary} with respect to the partition into arcs $\{ \Arc [\alpha], \alpha \in \cA \}$, that is by the sequence $(\ab{n})_{n \in \NN}$, where $\ab{n} \in \cA$ are such that  $\f^n (\xi) \in \Arc [\ab{n}]$ for any $n \in \NN$. We will call such sequence the \emph{boundary expansion} of $\xi$. Moreover, if $\theta = \theta(\xi)$ (see Remark~\ref{directions}),
in analogy with the continued fraction notation, we will write 
\[
	\theta \peq{[\ab{0}, \ab{1}, \dots]}. 
\]
When we write the above equality or say that $\xi$ has boundary expansion $(\ab{n})_{n \in \NN}$ we implicitely assume that $\xi=\xi(\theta)$ is non-cuspidal. 
 One can show that 
  the only restrictions on letters which can appear in a boundary expansion $(\ab{n})_{n \in \NN}$  is that $\alpha$ cannot be followed by $\overline{\alpha}$, that is
\begin{equation}\label{eqnobacktrack}
\ab{n+1}\not=\overline{\ab{n}}	 \qquad \text{ for any  $n\in\NN$}.
\end{equation}
We will call this property the \emph{no-backtracking condition}\footnote{See the very end of this section for the reason for the choice of this name}.
Boundary expansions can be defined also for cuspidal points (see Remark~\ref{boundary_cuspidal}) but are unique exactly for non-cuspidal points. Every  sequence in $\cA^{\NN}$ which satisfies the no-backtracking condition can be realized as a boundary expansion (of a cuspidal or non-cuspidal point).  

We will adopt the following notation. Given a sequence of letters $\ab{0}, \ab{1},  \dots, \ab{n}$, let us denote by
\[
	\Arc [\ab{0}, \ab{1},  \dots, \ab{n}] =\overline{ \Arc [\ab{0}] \cap \f^{-1}( \Arc [ \ab{1}] )\cap  \dots \cap \f^{-n}(\Arc [\ab{n}])}
\]
the closure of  set of points on $\partial\DD$ whose boundary expansion starts with  $\ab{0}, \ab{1},  \dots, \ab{n}$.
One can see that $\Arc[\ab{0},   \dots, \ab{n}]$ is a connected arc on $\partial \DD$ which is non-empty exactly when the sequence satisfies the no-backtracking condition~\eqref{eqnobacktrack}.
From the definition of $\f$, one can work out that
\beq \label{arcexpression}
	\Arc [\ab{0}, \ab{1},  \dots, \ab{n}] = G_{\ab{0}} \dots G_{\ab{n-1}} \Arc [\ab{n}].
\eeq 
Thus  two such arcs are \emph{nested} if one sequence contains the other as a beginning. For any fixed $n\in \NN$, the arcs of the form $\Arc[\ab{0}, \ab{1},  \dots, \ab{n}]$, where $\ab{0}, \ab{1},  \dots, \ab{n}$ vary over all possible sequences of $n$ letters in $\cA$ which satisfy the no-backtracking condition, will be called an \emph{arc of level} $n$.  
To produce the arcs of level $n+1$, each arc of level $n$ of the form $\Arc[\ab{0}, \ab{1},  \dots, \ab{n}]$ is partitioned into $2d-1$ arcs, each of which has the form $\Arc[\ab{0}, \ab{1},  \dots, \ab{n+1}]$  for $\ab{n+1} \in \cA\setminus \{\overline{\ab{n}}\}$.
Each one of these arcs corresponds to one of the arcs cut out by the sides of the ideal polygon $\ab{0}\ab{1}  \dots \ab{n} D$ and contained in the previous arc $\Arc[\ab{0}, \ab{1},  \dots, \ab{n}]$.  
One can show that if $\theta \peq{[\ab{0}, \ab{1}, \dots]}$ one has that 
\beq \label{shrink}
	\xi(\theta)=	\bigcap_{n\in\NN} G_{\ab{0}}\dots G_{\ab{n}} \Arc [\ab{n+1}].
\eeq
Thus, \emph{the cutting sequence of the ray which starts at the origin and ends at a non-cuspidal $\xi$ gives the entries of the boundary expansion of $\theta(\xi)$}.
Notice that the combinatorial no-backtracking condition~\eqref{eqnobacktrack} corresponds to the no-backtracking geometric phenomenon between hyperbolic geodesics we mentioned earlier.  

\section{Lagrange values via boundary expansions}\label{sec:Lagrangeboundary}
In this section we show that one can use boundary expansions to study Lagrange spectra of Veech translation surfaces. 
From now on, let $S$ be a fixed Veech translation surface and let $\cL:= \cL(S)$ denote its Lagrange spectrum. 

\subsection{Boundary expansions for Veech surfaces}
Let us consider a subgroup $\Gamma$ of finite index in the Veech group $V(S)$ that has no torsion elements, i.e.\ such that it does not contain elliptic elements. 
Then  $\Gamma$  admits a fundamental domain for the action on $\mathbb{D}$ which is an \emph{ideal} polygon $D$, that is an hyperbolic polygon having finitely many vertices $\xi_1, \dots, \xi_{2d}$ all lying on $\partial\DD$, see~\cite{Tukia}.
The domain $D$ is a finite cover of a fundamental domain for $V(S)$ and hence the induced tesselation of the hyperbolic plane has tiles which are finite union of the original tiles.
We will now use the tessellation given by $D$ to code geodesics on $\DD$ and to study flat geodesics on $S$. 
Let us stress that we do \emph{not} pass to a \emph{finite cover} of the surface $S$, which is fixed and has $V(S)$ as its Veech group.
We simply code geodesics in $\DD$ according to a tessellation of the disk that is more suited to our purposes than the one we would obtain from $V(S)$.

\begin{rem}
One can obtain an ideal polygon $D$ in a \emph{canonical} way, which is related to the flat geometry of the lattice surface $S$.
We recall that the \emph{spine} $\Pi$ of a translation surface is the subset of $\DD$ formed by the (isometry equivalence classes\footnote{We say that $S$ and $S'$ are \emph{isometric} if they are affinely equivalent by an affine map $\Psi:S \to S'$ with $D\Psi\in \operatorname{O} (2)$. The property of belonging to the spine is clearly invariant by isometry, so it is well defined on isometry classes of surfaces affinely isomorphic to $S$, which are in one to one correspondence to points in $\DD$, see for example~\cite{SmillieUlcigrai}.} of) surfaces which are affinely isomorphic to $S$ and that have two non-parallel minimal length saddle connections.
Then $\Pi$ is a deformation retract of $\DD$ (see~\cite{SmillieWeissTriangles}) and hence it is simply connected, which implies that $\Pi$ is a \emph{tree}. 
Its vertices are the surfaces with at least three, pairwise non-parallel, minimal length saddle connections.
If we consider the tessellation dual to $\Pi$, we obtain a tessellation of the hyperbolic disk into ideal polygons.
The spine is invariant under the action of the Veech group $V(S)$ but the tiles in the dual tessellation are permuted under its action.
Since all non-identity elements of $V(S)$ that fix a tile have to have finite order, we can quotient the Veech group to get a torsion-free group $\Gamma$ (that amounts to glueing some tiles of the tessellation together to obtain the ideal polygon domain $D$).
\end{rem}

From now on, we suppose that we have fixed a choice of a domain $D$, which is an ideal polygon with $2d$ sides.
As before, we will let $\cA_0$ denote an alphabet with $d$ symbols and let $\cA = \cA_0 \cup \overline{\cA_0}$. For each $\alpha \in \cA$ we choose a fixed representative $G_\alpha \in \SL(2,\RR)$ of the generator $G_\alpha \in \Gamma \subset \PSL(2,\RR)$. 
Thanks to the identification of the Teichm\"uller disk of $S$ and $T^1 \DD$ described in the previous section, we can use the boundary expansions to code affine deformations of the surface $S$ itself.

\subsection{Wedges associated to boundary expansions}\label{sec:wedges}
In this section, we first show how to associate to the boundary expansion a collection of saddle connections, which are obtained acting  by the linear action associated to boundary expansions on $d$ pairs of displacement vectors of saddle connections, that we call \emph{wedges}.
We show that, in order to evaluate large values of the Lagrange spectrum, it is enough to know the areas of these collection of saddle connections.
We will then show in \S~\ref{sec:renormalized_formula} that these areas can be expressed by a convenient formula. 

For any vertex $\xi_i$, $1\leq i \leq 2d$ of $D$ consider the corresponding direction $\theta_i:= \theta(\xi_i)$ on the translation surface $S$, see Remark~\ref{directions}. It is well known (see for example~\cite{HS}) that since $\xi_i$ corresponds to a cusp of $\DD/\Gamma$, $\theta_i$ is a parabolic direction on the Veech surface $S$ and hence the surface $S$ has a cylinder decomposition in direction $\theta_i$. 
Let $\cD$ be the collection of all displacement vectors of all saddle connections which belong to the boundaries of the cylinders of the cylinder decompositions in directions $\theta_i$, for $1\leq i \leq 2d$.  
This finite set of saddle connections generates all other saddle connections, in the sense that for any  saddle connection $u$ on $S$ there exists a saddle connection $v \in \cD$ and a reduced word $g \in \Gamma$ in the generators such that $u$ is the image of $ v$ under the linear action of $g$ on $S$.

In order to analyse small areas of saddle connections, it is enough to consider a smaller set of saddle connections.
Namely, for each cylinder decomposition, we pick only the shortest saddle connection which belongs to the cylinders boundaries.
It is convenient to pair these saddle connections in \emph{wedges} as follows.  
Recall that, for each side $e_\alpha$ in $D$ labelled by $\alpha \in \cA$, we call $\xi_\alpha^l$ and $\xi_\alpha^r$ the left and right endpoint of $e_\alpha$.
Let $\theta_\alpha^l$  and $\theta_\alpha^r$ the the corresponding pair of parabolic directions.

\begin{defi}[Basic wedges]
For any $\alpha\in \cA$, let  $v^r_\alpha$ and $v^l_\alpha$  be the displacement vectors of the shortest saddle connections  in the directions $\theta^r_\alpha$ and $\theta^l_\alpha$ respectively.
We will call \emph{wedge} and denote by $W_\alpha$ the basis of $\RR^2$ formed by the ordered pair of displacement vectors  $(v^r_\alpha, v^l_\alpha)$. 
We will identify $v^r_\alpha, v^l_\alpha$ with the corresponding column vectors  which have as entries the coordinates $(\Re(v^r_\alpha), \Im(v^r_\alpha))$  and $(\Re(v^l_\alpha), \Im(v^l_\alpha)) $ with respect to the standard base of $\RR^2$. 
We will often abuse the notation and denote by $W_\alpha$ also the matrix in $\GL(2, \RR)$ that has as columns the two vectors $(v^r_\alpha, v^l_\alpha) $, that is
\beq 
	W_\alpha = \begin{pmatrix} 
				\Re(v^r_\alpha) 	& 	\Re(v^l_\alpha) \\ 
				\Im(v^r_\alpha) 	& 	\Im(v^l_\alpha) 
			\end{pmatrix}.
\eeq
\end{defi}

Remark that $v_\alpha^r$ is the first vector and $v_\alpha^l$ is the second of the basis in order to give it positive orientation, consistently with the orientation of $\partial\DD$.
If $v$ is a displacement vector for some saddle connection on $S$ then $-v$ also is one.
Hence we define $-W_\alpha$ as the wedge formed by the vectors $-v_\alpha^r$ and $-v_\alpha^l$. 
Modulo replacing $S$ by some element in its $\SL (2, \RR)$-orbit, we assume that all the $W_\alpha$ are in the upper half-plane $\RR\times\RR_+$ and all the $-W_\alpha$ are in the lower half-plane $\RR\times\RR_-$.

We want to use the boundary expansion of directions in $\partial \DD$ to obtain a way of approximating directions in $\RR \PP^1$ with the directions of  a sequence of wedges (see Lemma~\ref{properties_boundary} below).
For $\alpha\in\cA$ the generators $G_\alpha$ of $\Gamma$ act linearly on $\RR^2$ as elements of $\slduer$.
For any $\alpha$ in $\cA$ consider the cone $\cW_\alpha$ in $\RR^2$ spanned by $W_\alpha$, that is the set of vectors
\[
	v=\pm(av_\alpha^r+bv_\alpha^l), \qquad \text{ with $a\geq0$ and $b\geq 0$.}
\]
Let $\cU_\alpha$ be the complement of $\cW_{\overline{\alpha}}$, namely $\RR^2\setminus\cW_{\overline{\alpha}}$, which is of also a cone in $\RR^2$.
By definition of the boundary expansion, we have
\[
	G_\alpha \cU_\alpha =\cW_\alpha, \qquad \text{ for any } \alpha,
\]
as shown in Figure~\ref{fig:Gaction}.

\begin{figure}
\centering
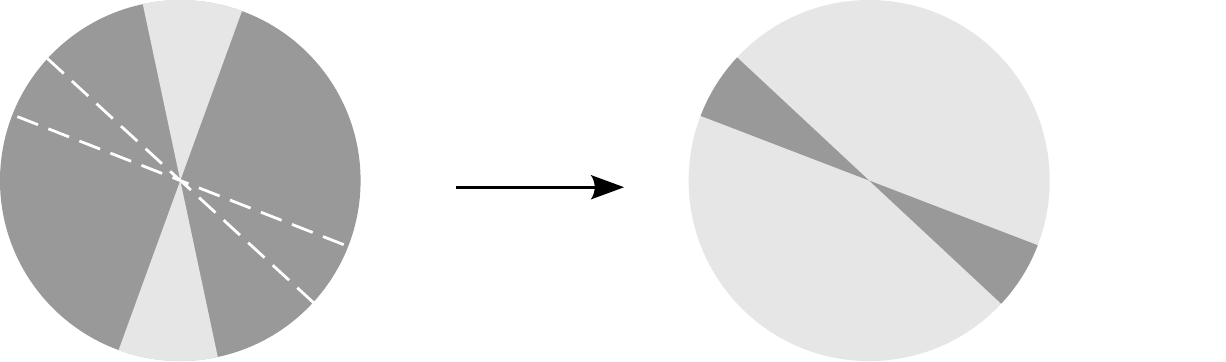
\caption{On the left, the cones $\cW_\alpha$, $\cU_\alpha$ and $\cW_{\overline{\alpha}}$. On the right, the image under $G_\alpha$ of $\cU_\alpha$ and $\cW_{\overline{\alpha}}$.}
\label{fig:Gaction}
\end{figure}

Consider a direction $\theta$ corresponding to a point $\xi(\theta)$ in $\partial\DD$ which is not a cusp and let $(\ab{n})_{n\in\NN}$ be the cutting sequence given by the boundary expansion, where $\ab{n}\in\cA$ for any $n$.
Define a sequence of wedges $\W{n}{\theta}$ setting
\begin{equation}\label{eq:wedgesfromboundaryexpansion}
\left\{ \begin{array}{ll}
	\W{0}{\theta} &:=W_{\ab{0}},\\
	\W{n}{\theta} &:= G_{\ab{0}}\dots G_{\ab{n-1}} W_{\ab{n}}, \qquad n \in \NN_+.
\end{array}
\right.
\end{equation}
Write $\W{n}{\theta}=\{ \vv{n}{r}, \vv{n}{l}\}$ and denote by $\theta(\vv{n}{r})$ and $\theta(\vv{n}{l})$ the directions in $\RR \PP^1$ of the displacement vectors $\vv{n}{r}$ and $\vv{n}{l}$ respectively.
We say that a \emph{saddle connection}  $v$ \emph{is in} the wedge $\W{n}{\theta}$  if $v =\vv{n}{r}$ or $v= \vv{n}{l}$ and moreover that a \emph{direction} $\eta$ \emph{belongs to a wedge} $(\vv{n}{r},\vv{n}{l})$ if $\theta(\vv{n}{r})\leq \eta\leq  \theta(\vv{n}{l})$.

For any $n$ let $\cW_n$ be the cone spanned by the wedge $\W{n}{\theta}$, that is the set of vectors of the form
\[
	v=\pm(a\vv{n}{r}+b\vv{n}{l}), \qquad \text{ with $a$, $b \geq 0$}.
\]
Observe that the cones $ \cW_n$ are \emph{nested}, that is we have $\cW_{n+1}\subset \cW_n$ for any $n$.
The same does not necessarily holds for the wedges $\W{n}{\theta}$, but this is not a problem, since we are interested only in the directions $\theta(\vv{n}{l})$ and $\theta(\vv{n}{r})$  of the vectors $\vv{n}{r}$ and $\vv{n}{l}$ in $\RR \PP ^1$ which converge monotonously to the direction $\theta \in \RR \PP^1$. 
More precisely, by definition of the wedges $\W{n}{\theta}$ and by the properties of the boundary expansion, one can prove the following.

\begin{lem}\label{properties_boundary}
For any $\theta$, the wedges $\W{n}{\theta}=(\vv{n}{r},\vv{n}{l})$ and the associated cones $\cW_n$ are such that
\begin{enumerate}
\item The cones $\cW_n$ 
 are \emph{nested} and all contain $\theta$, that is, for any $n \in \NN$ we have
\[
	-\frac{\pi}{2}\leq	\theta(\vv{n}{r}) \leq  \theta(\vv{n+1}{r}) \leq \theta \leq \theta(\vv{n+1}{l}) \leq  \theta(\vv{n}{l}) < \frac{\pi}{2}.
\]
\item The cones $\cW_{n}$ \emph{shrink} to $\theta$, i.e.\ the directions $\theta(\vv{n}{l})$ and $\theta(\vv{n}{r})$ both converge to $\theta$ as $n$ grows.
\end{enumerate}
\end{lem}

\subsection{Boundary expansion detects large values of the spectrum}
We will now show (see Theorem~\ref{boundary_see_largeL} below), that  the areas of the saddle connections in the wedges defined in the previous section allow to compute large values of the spectrum.
Let us first introduce some notation. Let $\cD_0 $ be the set of displacement vectors $v^l_\alpha$ or $v^r_\alpha$ appearing in the wedges $W_{\alpha}$ for $\alpha\in \cA$, namely
\[
	\cD_0 =  \{ v_\alpha^l, \alpha \in \cA\} = \{ v_\alpha^r, \alpha \in \cA\} .
\]
For $n\geq1$ let $\cD_n$ be the set of $u$ in $\RR^2$ of the form $u=\g_{\ab{1}}\dots \g_{\ab{n}}v$, where $\ab{1}, \dots \ab{n}$ is any sequence of letters in $\cA$ satisfying the no-backtracking condition~\eqref{eqnobacktrack} and $v\in\cD_0$. 
Remark that  $\cD_n$ is exactly the set of displacement vectors of saddle connections which can appear in a wedge of the form $\W{n}{\theta}(\theta)$ for some $\theta$. 
For any $n\geq0$ set
\beq\label{Mndef}
	M_n := \max \{ \abs{v} , v \in \cD_n \},
\eeq
where $\abs{v}$ denotes the Euclidean length of a vector $v$ in $\RR^2$.

\begin{thm}\label{boundary_see_largeL}
For any Veech surface $S$ there exists a constant $L_0=L_0(S)>0$ such that for any $\theta$ such that $L(\theta)>L_0$ we have 
\[
	L(\theta) = \frac{1}{\liminf_{n \to \infty} \areaW{\theta}{\W{n}{\theta}}}, \qquad \text{where} \quad  \areaW{\theta}{\W{n}{\theta}}= \min \{ \area{\theta}{\vv{n}{r}}, \area{\theta}{\vv{n}{l}} \}
\]
and  $(\W{n}{\theta})_{n \in \NN}$ is the sequence of wedges $\W{n}{\theta}=(\vv{n}{r},\vv{n}{l})$  associated to $\theta$.
Moreover, we can take
\[
L_0(S) = \frac{ 4 M_0 M_1 }{c(S)^2},
\]
where $c(S)$ is the constant defined below in Lemma~\ref{Veech:lemma} and $M_0$, $M_1$ are defined by~\eqref{Mndef}.
\end{thm}
We will actually show in the next section that one can define an acceleration of the boundary expansion map (see \S~\ref{sec:acceleration}) so that the corresponding subsequence of wedges still allows to compute large values of the spectrum (see Theorem~\ref{thmgaussdetects}). 
The rest of this section is devoted to the proof of Theorem~\ref{boundary_see_largeL}. The proof is based on the following two lemmas. 

\begin{lem}\label{Veech:lemma}
For a Veech surface $S$ there exists a positive constant $c=c(S)$ such that if $\sigma$ and $\gamma$ are two saddle connections of lengths $|\sigma|$ and $|\gamma|$ in directions $\theta(\sigma)$ and $\theta(\gamma)$ respectively, then we have
\[
	\abs{\theta(\gamma)-\theta(\sigma)} > \frac{c}{\abs{\gamma} \cdot \abs{\sigma}}.
\]
\end{lem}

\begin{proof}
Since $S$ is a Veech surface, then there exists a constant $M>1$ such that if $\gamma$ and $\gamma'$ are closed geodesics or saddle connections in the same direction, then we have $\abs{\gamma}<M\abs{\gamma'}$ and $\abs{\gamma'}<M\abs{\gamma}$.
Moreover there exists a constant $a>0$ such that the family of closed geodesics parallel to any given saddle connection $\gamma$ spans a cylinder $C$ with $\Area(C)>a$.
Let  $C_\gamma$ and $C_\sigma$ be cylinders in the two directions $\theta(\gamma)$ and $\theta(\sigma)$ of the two saddle connections $\gamma$ and $\sigma$. If we replace $C_\gamma$ and $C_\sigma$ by parallel cylinders $C_{\gamma'}$ and $C_{\sigma'}$ with core curves  the closed curves $\sigma',\gamma'$ parallel to $\sigma$ and $\gamma$ respectively, we can assume that $C_{\gamma'}$ and $C_{\sigma'}$ have non empty intersection.
Since cylinders have area bigger than $a$ then the width of $C_{\gamma'}$ is at least $a/\abs{\gamma'}$.
Since $\gamma'$ and $\sigma'$ are not parallel (they intersect since by assumption the corresponding cylinders intersect), than $\sigma'$ is not contained in $C_{\gamma'}$, see Figure~\ref{fig:cylinders}. It follows that
\[
	\abs{\sigma'} \cdot \abs{\sin(\theta(\sigma) - \theta(\gamma))}>\frac{a}{\abs{\gamma'}}
\]
and therefore
\[
	\abs{\theta(\sigma)-\theta(\gamma)} > \abs{\sin(\theta(\sigma)-\theta(\gamma))} %
	> \frac{a}{\abs{\gamma'}\cdot\abs{\sigma'}}\geq\frac{1}{M^2}\frac{a}{\abs{\gamma}\cdot\abs{\sigma}}.\qedhere
\]
\end{proof}

\begin{figure}
\centering
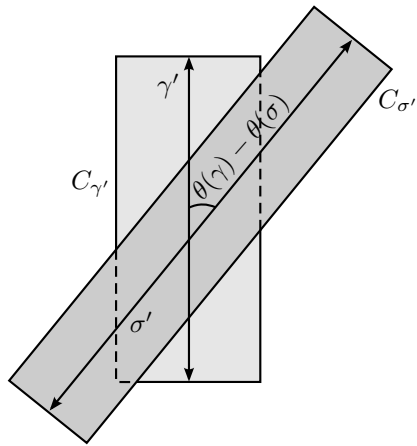
\caption{The two cylinders $C_{\gamma'}$ and $C_{\sigma'}$.}
\label{fig:cylinders}
\end{figure}

The following Lemma provides a useful way to bound areas of saddle connections. 
Fix a direction $\theta$ and consider the sequence of wedges $\W{n}{\theta}$ defined by Equation~\eqref{eq:wedgesfromboundaryexpansion}.
The area of a displacement vector $v$ of $S$ can be estimated in terms of the position of $v$ with respect to the wedges $\W{n}{\theta}$.
More precisely, we have the following.

\begin{lem}\label{area_control}
Fix a direction $\theta$ and consider the sequence  $(\W{n}{\theta})_{n \in \NN}$  associated to $\theta$ by the boundary expansion.
Suppose that $v$ is a displacement vector in direction $\theta(v)$ and consider integers $n$ and $N$ such that  either
\begin{equation}\label{eqcontrolarearight}
	\theta(\vv{n}{r} ) < \theta(v)< \theta(\vv{n+N}{r})<\theta,
\end{equation}
or
\begin{equation}\label{eqcontrolarealeft}
	\theta<\theta(\vv{n+N}{l})<\theta(v)<\theta(\vv{n}{l} ).
\end{equation}  
Then we have that 
\[
	\area{\theta}{v}\geq \frac{c(S)^2}{4 M_0 M_{N}}.
\]
\end{lem}

\begin{proof}
Let us assume that $\theta(\vv{n}{r} ) < \theta(v)< \theta(\vv{n+N}{r} )<\theta$.
The proof of the other case is analogous.
Let $\theta_\perp$ denote the direction orthogonal to $\theta$.
Remark that, since  all sides of $D$ have at most angular amplitude $\pi$, the wedge $\W{n}{\theta}$, which is contained in $\W{0}{\theta}$, has angular amplitude at most $\pi/2$.
Hence, since $\theta$ belongs to the wedge $\W{n}{\theta}$,  $\theta_\perp$ cannot belong to the same wedge $\W{n}{\theta}$, thus
\beq\label{orderperp}
	\theta_\perp < \theta(\vv{n}{r} ) < \theta(v) \quad \implies \quad \abs{\theta(v) -\theta_\perp} \geq \abs{\theta(\vv{n}{r} )-\theta_\perp}.
\eeq
 Let $\theta \peq{[\ab{0}, \dots, \ab{n}, \dots]}$ and let $\g:=\g_{\ab{n-1}}^{-1} \cdots \g_{\ab{0}}^{-1} $ be the element of $\Gamma$ such that $\g \W{n}{\theta}=W_{\ab{n}}$.
Let us renormalize by applying the element $\g$, that is consider the saddle connections $v':=\g v$, $v_0:=\g \vv{n}{r}$, $v_N:=\g \vv{n+N}{r}$ and let us denote respectively by $\theta(v')$, $\theta(v_0)$ and $ \theta(v_N)$  their directions. 
Let us now show that $\theta(v')$ is bounded away from the directions  $\theta'$, $\theta'_\perp$ defined by $\tan \theta'=\g * \tan\theta$ and  $\tan \theta_\perp'=\g * \tan \theta_\perp$.
In other words, $\theta'$ and $\theta'_\perp$ are directions of lines which are image of lines in direction $\theta$ and $\theta_\perp$ under $\g$.
 
Since $\g$ preserves the order of points on $\partial\DD$ and hence the order of the corresponding directions, it follows from~\eqref{eqcontrolarearight} and~\eqref{orderperp} that $\theta'_\perp < \theta(v_0)<  \theta(v')< \theta(v_N) < \theta'$, which implies
\beq \label{angleslowerbounds}
	 \abs{\theta' - \theta(v')} \geq \abs{\theta(v_N)-\theta(v')}
		\qquad \text{ and } \qquad
	 \abs{\theta(v') - \theta'_\perp} \geq \abs{\theta(v') -  \theta(v_0)}.
\eeq
Recall that $\area{\theta}{v} $ is the area of a rectangle $R$ with sides in directions $\theta$ and $ \theta_\perp$ and $v$ as a diagonal.
Since $\g$ acts linearly and preserves areas, $\area{\theta}{v} $ is also the area of the parallelogram  $\g \cdot R$ which has $v'$ as diagonal and sides in directions $\theta'$ and $\theta_\perp'$.
One can see that this area is bounded below as long as the angles $\abs{\theta(v')-\theta'}$ and $\abs{\theta(v')-\theta_\perp'}$ formed by the diagonal $v'$ with the sides of the parallelogram are bounded below. Indeed, calculating the area of the parallelogram by  the formula which gives the area of a triangle with a side of length $s$ adjacent to angles  $\psi, \phi$ as $(|{s}|^2 \sin\psi \sin \phi) /2 (\sin(\pi/2-\psi-\phi)$,  
 and using the trivial inequalities $\frac{\abs{x}}{2}\leq \abs{\sin(x)}\leq 1$ we get  that
\[
	\area{\theta}{v} = \frac{\abs{v'}^2 (\sin(\theta'-\theta(v') ) \sin(\theta(v')- \theta'_\perp)}{ \sin ( \frac{\pi}{2}-\theta'-\theta'_\perp)}
				 \geq \frac{\abs{v'}^2 \abs{\theta' -\theta(v')} \cdot \abs{\theta(v')-\theta_\perp'}}{4}.
\]
Thus, by~\eqref{angleslowerbounds}, using Lemma~\ref{Veech:lemma} and recalling the definition of the constants $M_0$, $M_N$ in~\eqref{Mndef}, we get
\[
\begin{split}
	\area{\theta}{v} 	&\geq \frac{\abs{v'}^2 \abs{\theta(v_N) - \theta(v')} \cdot \abs{\theta(v') - \theta(v_0)}}{4} 
				\geq \frac{c(S)^2 \abs{v'}^2 }{ 4  \abs{v_N} \abs{v'} \abs{v_0} \abs{v'}} 
				= \frac{c(S)^2}{ 4 \abs{v_N}  \abs{v_0} } \geq \frac{c(S)^2}{ 4 M_N M_0 }. \qedhere
\end{split}
\]
\end{proof}

We can now use Lemmas~\ref{Veech:lemma} and~\ref{area_control} to prove Theorem~\ref{boundary_see_largeL}.
\begin{proofof}{Theorem}{boundary_see_largeL}
Set $L_0:=  4 M_0 M_1/c^2(S)$ where $c(S)$ is the constant in Lemma~\ref{Veech:lemma} and $M_0$, $M_1$ are defined by~\eqref{Mndef}.
Let $\theta$ be such that $L(\theta)>L_0$ and let $(\W{n}{\theta})_{n \in \NN}$ be the sequence of associated wedges.   
Let us remark first that we always  have that $1/L(\theta) \leq \liminf_{n \to \infty} \area{\theta}{\W{n}{\theta}}$,  since $1/L(\theta)$ is computed considering the $\liminf$ on a larger set of saddle connections then the ones which belong to the wedges $\W{n}{\theta}$.
Thus it is enough to prove the converse inequality.

Let $(v_k)_{k\in \NN}$ be a sequence of saddle connections such that $(\area{\theta}{v_k})_{k\in \NN}$ converges to  $1/L(\theta)$ as $\Im_\theta(v_k)$ grows.  
We are going to show that there exists  $\overline{k}$ such that each $v_k$ with $k \geq \overline{k}$  belongs to $\W{n_k}{\theta}$ for some $n_k \in \NN$, so that we also have that  
\[
	\frac{1}{L(\theta)} = \liminf_{k \to \infty } \area{\theta}{v_k} 
		\geq \liminf_{\substack{k \to \infty \\ k\geq \overline{k}}} \area{\theta}{\W{n_k}{\theta}}
		\geq \liminf_{n \to \infty} \area{\theta}{\W{n}{\theta}}.
\]

Since $\Im_\theta(v_k)\to+\infty$, and displacement vectors of saddle connections form a discrete set, we must have that $\Re_\theta(v_k)\to 0$.
Thus the directions of all $v_k$, with $k$ sufficiently large, belong to the initial wedge $\W{0}{\theta}$.
According to Lemma~\ref{area_control}, if $v_k$ was neither equal to $\vv{n}{r}$ nor to $\vv{n}{l}$ for some $n$, we would have $\area{\theta}{v_k}>1/L_0$.
Since $\area{\theta}{v_k}\to1/L(\theta)$ and $L(\theta)>L_0$, this is not possible.
Hence $v_k$ is in some wedge $\W{\theta}{n_k}$ for some $n_k\in\NN$.
This concludes the proof.

\end{proofof}

\section{The cusp acceleration of the boundary expansion}\label{sec:acceleration}
We now define an acceleration of the boundary expansion. The acceleration is obtained by grouping together all steps which correspond to excursions in the same cusp,  in a similar way to how the Gauss map is obtained from the Farey map in the theory of classical continued fractions expansions. One can show that it is  sufficient to consider accelerated times to compute large values of the spectrum (see Theorem~\ref{thmgaussdetects}) and that one can bound areas of saddle connections in terms of the number of steps in the acceleration (see Proposition~\ref{propcontrolareaselemetarywords}).

\subsection{Cuspidal words and cuspidal sequences}
We first describe sequences that can taken to be  boundary expansions of the ideal vertices of $D$. 
\begin{defi}\label{defnestedarcs}
A \emph{left cuspidal word} (respectively a \emph{right cuspidal word}) is a word $\ab{0} \dots \ab{k}$ in the alphabet $\cA$ which satisfies the no-backtracking condition~\eqref{eqnobacktrack} and such that 
the $k+1$ arcs
\[
	\Arc[\ab{0}], \quad \Arc[\ab{0},\ab{1}], \quad \dots \quad  \Arc[\ab{0}, \dots, \ab{k-1}], \quad \Arc[\ab{0}, \dots, \ab{k}]
\]
all share as a  common left endpoint the left  endpoint $\xi_{\ab{0}}^l$ of $\Arc [\ab{0}]$ (respectively as right endpoint the  right endpoint $\xi_{\ab{0}}^r$ of $\Arc [\ab{0}]$), see Figure~\ref{fig:cuspidalword}. 
We simply write that $\ab{0} \dots \ab{k}$ is a \emph{cuspidal word} when left or right is not specified. 
We say that a sequence $(\ab{n})_{n \in \NN}$ is a \emph{cuspidal sequence} if any word of the form $\ab{0}\dots \ab{n}$ for $n \in \NN$ is a cuspidal word and that it is \emph{eventually cuspidal} if there exists $k \in \NN$ such that $(\ab{n+k})_{n \in \NN}$ is a {cuspidal sequence}. 
\end{defi}

\begin{figure}
\centering
\def\svgscale{0.8}
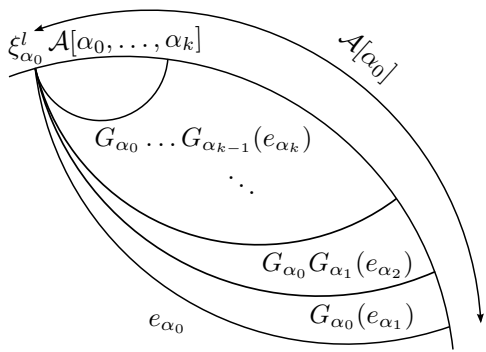
\caption{A left cuspidal word $\ab{0}\dots\ab{k}$.}
\label{fig:cuspidalword}
\end{figure}

Equivalently, $\ab{0} \dots \ab{k}$ is a left (right) cuspidal word exactly when the arc $\Arc [\ab{0}, \dots, \ab{k}] \subset \partial \DD$ has a vertex of $D$ as its left (respectively right) endpoint.
In \S~\ref{sec:combinatorialdescription} we show  that all cuspidal sequences are periodic (see Lemma~\ref{parabolic_in_cuspidal}) and we give an explicit combinatorial description of cuspidal words (see Lemma~\ref{lemcuspidalwords}).  
 
\begin{rem}\label{unique}
Remark that given an ideal vertex $\xi$, there is a \emph{unique} left (right) cuspidal word of length $k+1$ such that the arc $\Arc[\ab{0}, \dots, \ab{k}]$ has $\xi$ as left (right) endpoint.
Indeed, such word can be obtained as follows.
Let $\ab{0}$ be such that $\Arc[\ab{0}]$ has $\xi$ as its left (right) endpoint. 
For any $0\leq i < k$,  the arc  $\Arc [\ab{0},\dots, \ab{i}]$ of level $i$ is subdivided at level $i+1$ into $2d-1$ arcs of level $i+1$  and $\Arc[\ab{0},\dots, \ab{i+1}]$ is  the unique one which contains the left (respectively right) endpoint of $\Arc[\ab{0},\dots, \alpha_{i}]$.
\end{rem} 

\begin{rem}\label{boundary_cuspidal}
Boundary expansions can be defined for all points in the boundary $\partial D$ (and not only non-cuspidal points)  as follows.   \emph{Cuspidal sequences} can be taken by definition to be boundary expansions of vertices of the ideal polygon $D$.  More precisely, if $\xi$ is the left endpoint of $e_\alpha$ and the right  endpoint of $e_\beta$, that is $\xi=\xi^l_\alpha=\xi^r_\beta$, then $\xi$ has exactly two boundary expansions which are respectively given by the unique left boundary expansion starting with $\alpha$ and the unique right boundary expansion starting with $\beta$. Similarly,  \emph{eventually cuspidal} sequences can be taken to be boundary expansions of cuspidal points.
\end{rem} 

\subsection{Definition and properties of the cusp acceleration.}
Let us now define the \emph{cusp acceleration}.

\begin{defi}\label{defparabolicsteps}
Let $\theta \peq{[\ab{0}, \dots, \ab{n}, \dots ]}$  and let $(\W{n}{\theta})_{n \in \NN}$ be the sequence of wedges associated to $\theta$.
The \emph{sequence of accelerated times} is the sequence  $(n_k)_{k\in\NN}$ on integers $n_k$ defined as follows:
\[
	n_0:=0, \qquad n_{k+1} = \min \{ n > n_k \text{ such that $\ab{n_k} \dots \ab{n}$ is not a cuspidal word}\}.  
\]
The \emph{wedges associated to $\theta$ by the cusp acceleration} are the subsequence $(\W{n_k}{\theta})_{k \in \NN}$  of  the wedges $(\W{n}{\theta})_{n \in \NN} $ associated to $\theta$ corresponding to accelerated times. 
\end{defi}

By construction,  the integers $(n_k)_{k\in\NN}$ provide a decomposition of words into maximal cuspidal words, that is, for each $k \in \NN$,  $\ab{n_k} \dots \ab{n_{k+1}-1}$ is a cuspidal word and is the largest cuspidal word in  $\ab{n_k}, \ab{n_k+1}, \dots$ which starts with  $\ab{n_{k}}$. 
We remark that $n_{k+1}$ can be equal to $n_{k}+1$, i.e.\  maximal cuspidal words in the decomposition may have length one.

A description of the acceleration times in terms of wedges is given by the following Lemma, which shows that  the acceleration groups together exactly all steps in which one of the two displacement vectors in the wedges does not change.

\begin{lem}\label{lemelementarywords}
Let $\theta \peq{[\ab{0}, \dots, \ab{n},\dots]}$ and $(\W{n}{\theta})_{n \in \NN} $ be the sequence of wedges associated to $\theta$ by the boundary expansion.
Call $(n_k)_{k\in \NN}$ the sequence of accelerated times.
Then, for any time $k \in \NN$, exactly one of the following holds:
\begin{align}
	\vv{n}{r} &=\vv{n_k}{r}, \text{ for all $n_k \leq n < n_{k+1}$, and $\theta(\vv{n_{k+1}}{r}) \neq \theta(\vv{n_{k}}{r})$,}  \label{leftcase} \\
	\vv{n}{l} &=\vv{n_k}{l}, \text{ for all $n_k \leq n < n_{k+1}$, and $\theta(\vv{n_{k+1}}{l}) \neq \theta(\vv{n_{k}}{l})$.} \label{rightcase}
\end{align}
More precisely, if $\ab{n_k}\dots\ab{n_{k+1}}$ is a right cuspidal word, then~\eqref{leftcase} holds, and if  $\ab{n_k}\dots\ab{n_{k+1}}$ is a left cuspidal word, then~\eqref{rightcase} holds.
\end{lem}

\begin{proof}
By definition of accelerated times, $\ab{n_k} \dots \ab{n_{k+1}-1}$ is a maximal cuspidal word.
Let us show that if it is a right cuspidal word then~\eqref{leftcase} holds. 
Let us first remark that if we prove the equalities in~\eqref{leftcase}, namely $\vv{n}{r}=\vv{n_k}{r}$ for all $n_k < n < n_{k+1}$, then the inequality $\theta(\vv{n_{k+1}}{r}) \neq \theta(\vv{n_{k}}{r})$ in~\eqref{leftcase} follows from maximality,  since if it failed  then we would have $\xi^l_{n_k} = \dots = \xi^l_{n_{k+1}-1}=  \xi^l_{n_{k+1}}$ and $\ab{n_k} \cdots \ab{n_{k+1}-1} \ab{n_{k+1}}$ would also be cuspidal, contradicting the definition of $n_{k+1}$.  
Let us now show that  it is enough to prove that
\begin{equation}\label{k0case}
\vv{n}{r}=\vv{0}{r}, \qquad \text{ for } 0 < n < n_{1},
\end{equation}
that is to prove the equalities in~\eqref{leftcase} for the case $k=0$, recalling that $n_0=0$.  
Consider the element $G_k:= G_{\ab{0}} \cdots G_{\ab{n_k-1}}$ and  apply $G_k^{-1}$ to the  wedges $ \W{n}{\theta}$ with $n> n_k$, to obtain the wedges 
\beq \label{wedgesFn}
	G_k^{-1} \W{n}{\theta} =  \left(G_{\ab{n_k-1}}^{-1}\dots G_{\ab{0}}^{-1}\right)  G_{\ab{0}} \dots G_{\ab{n-1}} W_{\ab{n}} = 
		G_{\ab{n_k}} G_{\ab{n_k+1}} \dots G_{\ab{n-1 }} W_{\ab{n}},  \quad n>n_k. 
\eeq
Thus, since $G_k$ is invertible and hence acts bijectively on $\RR^2$, to prove that  $\vv{n}{r}=\vv{n_k}{r}$ for some $n > n_k$ it is enough to prove that the vector $G_k^{-1} v_r^{(n)}$  in~\eqref{wedgesFn} is equal to  $G_k^{-1}v_r^{(n_k)}=v_{\ab{n_k}}^r$. 
Remark  now that $\f$ acts as a shift on boundary expansions, so that $\f^{n_k}(\theta) \peq{ [\ab{n_k}, \ab{n_k+1}, \dots ]}$.
So, $W_{\ab{n_k}}$ and the wedges in~\eqref{wedgesFn} are the wedges $(\W{m}{\f^{n_k}(\theta)})_{ \in \NN}$ associated to $\f^{n_k}(\theta)$ by the boundary expansion. Hence, by replacing $\theta$ with $\f^{n_k}(\theta)$, one reduces to proving~\eqref{k0case}. Fix now $0<n<n_1$. Since
\[
	\W{n}{\theta}= \g_{\ab{0}}\dots \g_{\ab{n-1}}W_{{\ab{n}}}, 
\]
$\vv{n}{r}$ is the image by a linear map of the right vector $v^r_{\ab{n}}$ in $ W_{\ab{n}}$, that by definition is the shortest vector in direction $\theta(v^r_{\ab{n}})$. It follows that  $\vv{n}{r}$  is the shortest vector in direction $\theta(\vv{n}{r})$. 
On the other hand,  since $\ab{0}, \dots, \ab{n}$ is a right cuspidal word, by Definition~\ref{defnestedarcs} we have that $\xi_n^r = \xi_0^r $ and hence also  $\theta(\vv{n}{r})=\theta(\vv{0}{r})$. Thus, since by definition $\vv{0}{r}$ is the shortest displacement vector in direction $\theta(\vv{0}{r})$, by uniqueness of the shortest displacement vector of a saddle connection in  direction $\theta(\vv{n}{r}) = \theta(\vv{0}{r})$, we must have $\vv{n}{r}= \vv{0}{r}$.
This concludes the proof of~\eqref{leftcase}. The case of a left cuspidal word leads to~\eqref{rightcase} and is proved analogously.
\end{proof}

One can show that accelerated times are sufficient to evaluate large values of the spectrum.

\begin{thm}\label{thmgaussdetects}
Let $S$ be a Veech surface and assume that $\theta$ is such that	$L(\theta)>{4M_0M_2}/{c(S)^2}.$
Let $(\W{n_k}{\theta})_{k\in \NN}$ be the corresponding sequence of wedges at  the accelerated times introduced in Definition~\ref{defparabolicsteps}.
Then, denoting as before  $\areaW{\theta}{\W{n_k}{\theta}}= \min \left\{ \area{\theta}{\vv{n_k}{r}}, \area{\theta}{\vv{n_k}{l}} \right\}$, we have
\[
	L(\theta)= \frac{1}{\liminf_{k\to\infty} \areaW{\theta}{\W{n_k}{\theta}}}. 
\]
\end{thm}

We do not include the proof of this result since we will not use it. The proof can be easily deduced by the estimates in the following Lemma, that we will use in the proof of the existence of the Hall ray.

\begin{prop}\label{propcontrolareaselemetarywords}
Let $(\W{n}{\theta})_{n \in \NN}$ be the sequence of wedges associated to a direction $\theta$ and let $(n_k)_{k\in \NN}$ be the sequence of accelerated times. Let $c=c(S)$ be as in Lemma~\ref{Veech:lemma} and 
let $N_k:=n_{k+1}-n_k$ be the lengths of the the $k^\text{th}$ maximal cuspidal word. 
For any $k \in \mathbb{N}$, either~\eqref{leftcase} holds, in which case  we have 
\begin{equation}\label{leftareas}
	\area{\theta}{\vv{n}{r}}  =\area{\theta}{\vv{n_k}{r}}> \frac{c^2}{4M_0M_{N_k+2}}, 	\qquad \text{ for }  n_k\leq n < n_{k+1}, 
\end{equation}
and
\begin{subequations}\label{allrightareas}
	\begin{align}
&	\area{\theta}{\vv{n_k}{l}}  >\frac{c^2}{4M_0M_{N_{k-1}+3}}, \label{firstrightarea} \\
&	\area{\theta}{\vv{n}{l}}  >\frac{c^2}{4M_0M_{2}}, 			&  &\text{ for }  n_k< n< n_{k+1}-1, \label{rightareas} \\ 
&	\area{\theta}{\vv{n_{k+1}-1}{l}}  >\frac{c^2}{4M_0M_{N_{k+1}+2}},  \label{lastrightarea}
\end{align}
\end{subequations}
or~\eqref{rightcase} holds, in which case the same inequalities hold  with the role of $r$ and $l$ exchanged.    
\end{prop}
\begin{proof}
Fix $k \in \NN$.
By Lemma~\ref{lemelementarywords}, either~\eqref{leftcase} or~\eqref{rightcase} hold.
Let us assume that~\eqref{leftcase} holds and prove the inequalities~\eqref{leftareas} and~\eqref{allrightareas}.
The case when~\eqref{rightcase} holds is analogous. Remark that, since for any $n$ the wedges $\W{n}{\theta}$ and $\W{n+1}{\theta}$ differ by at least one vector, by~\eqref{leftcase}  for any $n_k < n< n_{k+1}-1$ we have
\begin{equation}\label{otherinequalities}
	\theta\big(\vv{n-1}{l}\big)< \theta\big(\vv{n}{l}\big)< \theta\big(\vv{n+1}{l}\big),
\end{equation}
while one could have that $\theta\big(\vv{n_{k+1}-1}{l}\big)= \theta\big(\vv{n_{k+1}}{l}\big)$.
The inequality~\eqref{rightareas} for $n_k < n< n_{k+1}-1$ hence follows immediately from~\eqref{otherinequalities} by applying Lemma~\ref{area_control} with $N=2$. 
In order to prove~\eqref{leftareas}, observe first that by~\eqref{leftcase} $\vv{n}{r}=\vv{n_k}{r}$ for $n_k\leq n< n_{k+1}$.
We claim  that we have
\begin{equation}\label{inequalitiesleft}
	\theta \big(\vv{n_k-2}{r} \big)<\theta\big(\vv{n}{r}\big)<\theta\big(\vv{n_{k+1}}{r}\big).
\end{equation}
From~\eqref{inequalitiesleft}, applying Lemma~\ref{area_control} with $n={n_k-2}$ and $N=N_k+2=n_{k+1}- (n_k-2)$, we immediately get~\eqref{leftareas}. We are hence left to prove~\eqref{inequalitiesleft}: 
the last inequality is part of~\eqref{leftcase}, while the first inequality follows by remarking that  we cannot have $\theta \big(\vv{n_k-2}{r} \big)=\theta \big(\vv{n_k-1}{r} \big) = \theta\big(\vv{n_k}{r}\big)$ since this would contradict the definition of $n_{k}$. More in detail, if $\theta\big(\vv{n_k-1}{r} \big) < \theta\big(\vv{n_k}{r}\big)$ then $\theta \big(\vv{n_k-2}{r} \big)\leq  \theta\big(\vv{n_k-1}{r} \big) < \theta\big(\vv{n_k}{r}\big)$ and we are done.
If $\theta\big(\vv{n_k-1}{r} \big) = \theta\big(\vv{n_k}{r}\big)$ then~\eqref{leftcase} does not hold for ${k-1}$.
Thus~\eqref{rightcase} for ${k-1}$ yields $\theta \big(\vv{n_k-2}{l} \big)=  \theta\big(\vv{n_k-1}{l} \big)$ and so $\theta \big(\vv{n_k-2}{r} \big)<  \big(\vv{n_k-1}{r} \big)$. 

Finally, we obtain the bounds in~\eqref{firstrightarea} and in~\eqref{lastrightarea} by applying twice Lemma~\ref{area_control} to the inequalities
\[
	\theta\big(\vv{n_{k-1}-2}{l}\big)<\theta\big(\vv{n_{k}}{l}\big)<\theta\big(\vv{n_{k}+1}{l}\big)
	\quad \text{ and } \quad
	\theta\big(\vv{n_{k+1}-2}{l}\big)<\theta\big(\vv{n_{k+1}-1}{l}\big)<\theta\big(\vv{n_{k+2}}{l}\big),
\]
which can be proved reasoning as above and  using the definition $n_{k-1}$, $n_{k+1}$ and $n_{k+2}$.
\end{proof}

\subsection{Combinatorial description of cuspidal words}\label{sec:combinatorialdescription}
In this section  we  give a combinatorial descriptions of cuspidal words. 

\begin{lem}\label{lemcuspidalwords} Assume that $\alpha_0, \alpha_1, \dots,\alpha_k$  which satisfies the no-backtracking Condition~\eqref{eqnobacktrack}.  

The word $\alpha_0 \alpha_1 \dots \alpha_k$  is a \emph{left cuspidal word} if and only if
\beq\label{left_cuspidal}
	\xi_{\alpha_i}^l = \g_{\alpha_i}(\xi_{\alpha_{i+1}}^l), \qquad \text{for $0 \leq i \leq k-1$.}
\eeq
Similarly, $\alpha_0 \alpha_1 \dots \alpha_k$  is a \emph{right cuspidal word} if and only if
\beq\label{right_cuspidal}
	\xi_{\alpha_i}^r  = \g_{\alpha_i}(\xi_{\alpha_{i+1}}^r) , \qquad \text{for $0 \leq i \leq k-1$.}
\eeq
\end{lem}

\begin{proof}
Assume that $\alpha_0, \dots, \alpha_k$ satisfy the no-backtracking condition~\eqref{eqnobacktrack}.
For $1 \leq i < k$, we write $\Arc [\alpha_0, \dots, \alpha_{i}] = \g_{\alpha_0} \dots \g_{\alpha_{i-1}} \Arc[\alpha_{i}]$.
Then, by Definition~\ref{defnestedarcs}, we know that $\alpha_0 \dots \alpha_k$ is a left cuspidal word if and only if, for any $0\leq i \leq k-1$, the pair of arcs 
\[
	\Arc [\alpha_0, \dots, \alpha_{i}] \quad \text{and} \quad \Arc [\alpha_0, \dots, \alpha_{i+1}],
\]
share a common left endpoint. Thus, applying $( \g_{\alpha_0} \dots \g_{\alpha_{i-1}})^{-1}$, we see that this equivalently means that for every $0\leq i < k $ the pair of arcs 
 \[
	\Arc [\alpha_{i}] =( \g_{\alpha_0} \dots \g_{\alpha_{i-1}})^{-1} \Arc [\alpha_0, \dots, \alpha_i], \quad  \g_{\alpha_{i}} \Arc [\alpha_{i+1}]=
{( \g_{\alpha_0} \dots \g_{\alpha_{i-1}})}^{-1} \Arc [\alpha_0, \dots, \alpha_{i+1}] 
\]
share an endpoint.
In particular, since the left endpoints of $\Arc [\alpha_{i}]$ and $\Arc [\alpha_{i+1}]$ are respectively $\xi_{\alpha_i}^l$ and $\xi_{\alpha_{i+1}}^l$ and since $\alpha_{i+1}\neq \overline{\alpha}_i$, we see that $ \g_{\alpha_{i}}$ maps left endpoints of arcs to left endpoints of arcs. 
Thus $\alpha_0, \dots \alpha_k$ is a left parabolic word if and only if~\eqref{left_cuspidal} holds.
The proof for right cuspidal words is analogous.
\end{proof}

We now show that cuspidal sequences are simply obtained by repeating periodically words which correspond to parabolic elements.   

\begin{lem}\label{parabolic_in_cuspidal}
Given any cuspidal sequence $(\ab{n})_{n \in \NN}  $ there exists an integer $k>0$ such that
\begin{enumerate}
	\item The group element $\g= \g_{\alpha_0}\dots \g_{\alpha_{k-1}} \in \Gamma$ associated to $\alpha_0 \cdots \alpha_{k-1}$ is a parabolic linear transformation. If $(\ab{n})_{n \in \NN}  $ is right cuspidal, $G$ fixes the vector $v^r_{\alpha_0}$ in direction $\theta(v^r_{\alpha_0}) =	\theta(v^l_{\overline{\alpha_{k-1}}})$, while if it is left cuspidal $G$ fixes $v^l_{\alpha_0}$ in direction $ \theta(v^l_{\alpha_0}) =	\theta(v^r_{\overline{\alpha_{k-1}}})$. 
\item The infinite cuspidal word $\alpha_0 \alpha_1 \cdots \alpha_n \cdots$ is obtained by repeating periodically the word $\alpha_0 \dots \alpha_{k-1}$, that is, using the notation $n \bmod k$ for the remainder of the division of $n$ by $k$, we can write
\[
	\ab{n} = \alpha_{ n \bmod k} \qquad \text{for all } n \in \NN.
\]
\end{enumerate}
\end{lem}
\begin{notation}\label{parabolicword_def}
Given  a cuspidal sequence $(\ab{n})_{n \in \NN} $, we will say that $k$ given by Lemma~\ref{parabolic_in_cuspidal} is the \emph{period} of the cuspidal sequence and that    $\alpha_0, \dots, \alpha_k$ is the \emph{parabolic word} associated to $(\ab{n})_{n \in \NN} $.
 \end{notation}
\begin{proof}
Let us assume that $(\ab{n})_{n \in \NN}$ is left cuspidal. The other case is analogous. 
Let $k \in \NN$ be the smallest integer such that $\ab{k}=\ab{0}$.
Since the word $\alpha_0 \dots \alpha_{k-1} \alpha_0$ coincide by assumption with the initial word $\ab{0} \dots \ab{k-1} \ab{k}$ of $(\ab{n})_{n \in \NN}$ and hence it is cuspidal, it follows from Lemma~\ref{lemcuspidalwords} that 
\beq \label{parabolic_chain}
	\begin{split}
		\xi_{\alpha_i}^l  &= \g_{\alpha_i}(\xi_{\alpha_{i+1}}^l) , \qquad \text{ for $ 0\leq i \leq k-1$,} \\
		\xi_{\alpha_k}^l & = \g_{\alpha_k}(\xi_{\alpha_0}^l).
	\end{split}
\eeq 

Let us consider  $\g= \g_{\alpha_0}\dots \g_{\alpha_{k-1}} \in \Gamma$.
It follows from~\eqref{parabolic_chain} that $\g$ fixes the point $\xi_{\alpha_0}^l$ in $\partial\DD$ and hence the direction $\theta(v^l_{\alpha_0})$, since $v^l_{\alpha_0}$ is by definition a vector the direction $\theta(\xi_{\alpha_0}^l)$ which correspond to $\xi_{\alpha_0}^l$.  Since $v^l_{\alpha_0}$ is the shortest displacement vector in direction $\theta(v^l_{\alpha_0})$,  linear maps send the shortest vector in a given direction to the shortest vector in the image direction and $G $ fixes the direction $\theta(\xi_{\alpha_0}^l)$, it follows that  $\g v^l_{\alpha_0}=  v^l_{\alpha_0}$. Finally, to show that $\theta(v^l_{\overline{\alpha_0}})=\theta(v^r_{\alpha_k})$, remark that,  by using  $\g_{\alpha_k} (\xi_{\alpha_0}^l) =  \xi_{\alpha_k}^l$ in~\eqref{parabolic_chain} and~\eqref{action_g_on_labels}, we get
\[
\xi_{\alpha_0}^l = \g_{\alpha_k}^{-1} (\xi_{\alpha_k}^l) = \g_{\overline{\alpha_k}} (\xi_{\alpha_k}^l ) =\xi_{\overline{\alpha_k}}^r.
\]
Thus Part $(1)$ of the Lemma is proved.
To show Part $(2)$, let us remark that by~\eqref{arcexpression} the element $\g$ defined above is such that
\[
	\g \Arc [\alpha_0] = \g_{\alpha_0}\dots \g_{\alpha_{k-1}} \Arc [\alpha_0] = \Arc [ \alpha_0, \dots , \alpha_{k-1}, \alpha_0] = \Arc [ \ab{0}, \dots , \ab{k-1}, \ab{0}].
\]
Moreover, since $\alpha_0, \dots, \alpha_{k-1}$ is a cuspidal word, by Definition~\ref{defnestedarcs} we know that $\Arc [\alpha_0,\dots , \alpha_i]$ all share a common endpoint for $0\leq i < k$.
Thus it follows that  also 
\[
	\g \Arc [ \alpha_0, \dots, \alpha_i] =  \Arc [\alpha_0, \dots, \alpha_{k-1}, \alpha_0, \dots , \alpha_i ], \qquad 0\leq i < k 
\]
all share the same endpoint, which is the same endpoint of $\Arc [ \ab{0}, \dots , \ab{k-1}, \ab{0}]$.
Thus also the word $\alpha_0, \dots , \alpha_{k-1}, \alpha_0, \dots, \alpha_{k-1}$ is   left  cuspidal.
Since there is a  unique left cuspidal word starting with $\alpha_0$ (see Remark~\ref{unique}),  it follows that $\alpha_n=\alpha_{n \mod k}$ for all $k \leq n < 2k$. Using this argument with powers $\g^l$ of $\g$, one can prove by induction on $l$ the desired conclusion.
\end{proof}

\begin{rem}\label{cuspdescription} If $\g_\alpha=\g_{\overline{\alpha}}^{-1}$ is a parabolic generator of $\Gamma$, the two sides $e_{\alpha}, e_{\overline{\alpha}}$ share a common vertex, say $\xi = \xi^r_\alpha = \xi^l_{\overline{\alpha}}$. In this case
\[
	\g_\alpha(\xi^l_{\overline{\alpha}}) = \xi^l_{\overline{\alpha}} = \xi^r_\alpha=\g_{\overline{\alpha}} (\xi^r_\alpha)
\]
and the length-one words $\alpha$ and $\overline{\alpha}$ identify a cusp  with fixed parabolic direction $\theta(v^l_{\overline{\alpha}})=\theta(v^r_\alpha)$. More in general,  the cusps of $\DD/\Gamma$ are in bijection with the set of parabolic words, modulo the operations of cyclical permutation of the entries and inversion $\alpha_0 \dots \alpha_{k}\mapsto \overline{\alpha_{k}} \dots \overline{\alpha_0}$. 
\end{rem}

\section{Existence of a Hall's ray}\label{sec:Hall}
In this section we prove our main result, Theorem~\ref{thm:HallVeech}, that is the existence of a Hall ray for any Veech surface.
We first state and explain two results which are needed in the proof.
The first one is a formula for areas of wedges using the boundary expansions (see Theorem~\ref{thmrenormalizedformula}  and Corollary~\ref{correnormalizedformula}),   which has a similar form to the formula~\eqref{eq:classical_formula} for values of the classical Lagrange spectrum  in terms of the continued fraction entries. This formula 
allows to express large values in the spectrum as a sum of two Cantor sets.  A result proved in the original paper by Hall is that, under a technical condition, the sum of two Cantor sets contains an interval. 
The second result needed (Proposition~\ref{propsumcantorsets}) is that this condition is satisfied for the Cantor sets given by the formula in Theorem~\ref{thmrenormalizedformula}. After stating these two results, respectively in \S~\ref{subsec:formula} and \S~\ref{subsec:sumCantor}, we present in \S~\ref{subsec:finalarguments} the arguments which allow to construct the Hall's ray and hence prove Theorem~\ref{thm:HallVeech}.  


\subsection{A formula for areas of wedges}\label{subsec:formula}
 In order to state the formula for areas of wedges, we need to introduce some notation. Recall that  we write $A \in \pslduer$ to denote the equivalence class in  $\pslduer$ of a  matrix $A \in \slduer$. Consider the matrices $G_\alpha \in \pslduer$, $\alpha \in \cA$, introduced in~\S~\ref{subsection:boundary} as generators of $\Gamma$ and recall that we use the symbol $W_\alpha$  for the matrix
\[
	W_\alpha = \begin{pmatrix} 
				\Re(v^r_\alpha) 	& 	\Re(v^l_\alpha) \\ 
				\Im(v^r_\alpha) 	& 	\Im(v^l_\alpha) 
			\end{pmatrix}.
\]
For any fixed letter $\alpha\in\cA$, define the family of matrices
\begin{equation}\label{Amatricesdef}
	A_\beta^\alpha:= (W_\alpha^{-1}G_\alpha)\, G_\beta \, (W_\alpha^{-1}G_\alpha )^{-1}=W_\alpha^{-1}G_\alpha G_\beta G_\alpha^{-1} W_\alpha ,
 \qquad \beta \in \cA .\end{equation}
We will also need a second family of matrices. If $A^T$ as usual denotes the transpose matrix of $A$, set 
\begin{equation}\label{Bmatricesdef}
	B_\beta^\alpha:=(RW_\alpha^T) \, G_\beta^T \, (R W_\alpha^T)^{-1} = RW_\alpha^T G_\beta^T  (W_\alpha^T)^{-1} R,\qquad \beta\in \cA, \qquad \text{where}\  R=\left(\begin{smallmatrix}0&1 \\ 1&0\end{smallmatrix}\right). 
\end{equation}
Let us now define two \emph{continued fractions} based on these two families of matrices as follows.   Recall  that  we denote by  $x \mapsto A * x $ the action of $A$ on $\RR$ by M\"obius transformations.  Let $\alpha$ be a letter in the alphabet $\cA$ and assume  $\alpha, \ab{0}, \ab{1} ,\dots, \ab{n}$  satisfies the no-backtracking condition~\eqref{eqnobacktrack}. The  (finite) \emph{$\alpha$-forward continued fraction} associated to the word $\ab{0} \ab{1}\dots\ab{n}$ is 
\begin{equation}\label{forwardCFdef}
	[\ab{0},\ab{1}, \dots,\ab{n}]_\alpha^+:=A_{\ab{0}}^\alpha A_{\ab{1}}^\alpha \cdots A_{\ab{n-1}}^\alpha \left( W_{\alpha}^{-1} G_\alpha W_\ab{n}\right) * \infty.
\end{equation}
Similarly, the (finite) \emph{$\alpha$-backwards continued fraction} associated to the word $\ab{0}\ab{1}\dots\ab{n}$ is
\begin{equation}\label{backwardCFdef}
	[\ab{0},\ab{1}, \dots,\ab{n}]_\alpha^-:=B_{\ab{0}}^\alpha B_{\ab{1}}^\alpha \cdots B_{\ab{n}}^\alpha \left(R W_\alpha^T (W_{\ab{n}}^T)^{-1} R\right) *\infty.
\end{equation}
We will  show in the next section (see Lemma~\ref{lemconvergencecontinuedfraction}) that, having fixed an infinite word $(\ab{n})_{n\in\NN}$ such that $\ab{0}\neq \overline{\alpha}$, both the previous expressions converge as $n $ tends to infinity. Hence the following quantities, which will be called respectively the (infinite) $\alpha$-\emph{forward} and $\alpha$-\emph{backward continued fractions} with entries $(\ab{n})_{n\in\NN}$,  are well defined:
\begin{align*}
	[\ab{0}, \ab{1},\dots,\ab{n},\dots]_\alpha^+ &:= \lim_{n\to\infty} [\ab{0}, \ab{1},\dots,\ab{n}]_\alpha^+,\\
	[\ab{0}, \ab{1},\dots,\ab{n},\dots]_\alpha^- &:= \lim_{n\to\infty} [\ab{0}, \ab{1},\dots,\ab{n}]_\alpha^-.
\end{align*}
We can now state the formula which expresses areas of saddle connections in the wedges using $\alpha$-backward and forward continued fractions.

\begin{thm}[Continued fraction formula]\label{thmrenormalizedformula}
Let $(\ab{n})_{n \in \NN}$ be the boundary expansion of a direction $\theta$.
Given two sequences of numbers $x_n$ and $y_n$, let us write $x_n\sim y_n$ if $|x_n-y_n | \to 0$ as $n\to\infty$. 
We have
\begin{align*}
	\frac{1}{\area{\theta}{\vv{n}{r} }} &\sim \frac{1}{\det (W_{\ab{n}})} \left( [\ab{n-1},\dots,\ab{0}]_{\ab{n}}^- +[\ab{n+1},\ab{n+2},\dots]_{\ab{n}}^+\right),\\
	\frac{1}{\area{\theta}{\vv{n}{l} }} &\sim \frac{1}{\det (W_{\ab{n}})} \left( \frac{1}{[\ab{n-1},\dots,\ab{0}]_{\ab{n}}^-} %
							+ \frac{1}{ [\ab{n+1},\ab{n+2},\dots]_{\ab{n}}^+} \right).
\end{align*}
\end{thm}
The proof of Theorem~\ref{thmrenormalizedformula} is given in \S~\ref{sec:renormalized_formula}, where  
we also deduce the following Corollary. 


\begin{cor}\label{correnormalizedformula}
Let $(\ab{n})$ be the boundary expansion of a direction $\theta$. For any $\alpha \in \cA$, there exists $\mu_\alpha>0$ such that the following holds. Let $\alpha \gamma_1 \cdots \gamma_{p-1}$ be the unique right cuspidal word which starts with $\alpha$.  Assume that there are arbitrarily large $n$ such that
\[
	\ab{n}=\alpha, \ab{n+1}=\gamma_1,\cdots,\ab{n+p-1}=\gamma_{p-1},  \ab{n+p}=\alpha.
\]
For such times\footnote{Since we require the extra condition $\ab{n+p}=\alpha$, that is the finite word $\ab{n}\dots\ab{n+p}$ is part of a bigger cuspidal word, we are looking at areas of wedges at instants which do not necessarily coincide with the instants of the parabolic acceleration.} we have
\[
	\frac{1}{\area{\theta}{\vv{n}{r} }} \sim \frac{1}{\det (W_{\alpha})} \left( [\ab{n-1},\dots,\ab{0}]_{\alpha}^- +\mu_\alpha %
							+[\ab{n+p+1},\ab{n+p+2},\dots]_{\alpha}^+\right).
\]
\end{cor}
Remark that this formula is very similar to the classical formula~\eqref{eq:classical_formula}, in the sense that it is the sum of a term which depends only on the \emph{past} of the symbolic coding and a term which depend only on the \emph{future}. 
Moreover, the constant $\mu_\alpha$ in the statement is the shear of the parabolic matrix given at point~(3) of Lemma~\ref{propertiesmatrices}


\subsection{Sums of Cantor sets}\label{subsec:sumCantor}
The formula in the previous section, as in the classical case, allows to express large values in the spectrum as sum of the following two Cantor sets.  Let us now fix an arbitrary positive integer $N$.
We call $\cK_N$ the set of all sequences $(\ab{n})_{n \in \NN}$ which satisfy the no-backtracking condition~\eqref{eqnobacktrack} and do not contain any cuspidal word of length $N$. 
Fix a letter $\alpha\in\cA$ and call $\cK_{N}^{\alpha}$ the subset of sequences  $(\ab{n})_{n \in \NN} \in \cK_N$ whose first letter $\ab{0} \neq \overline{\alpha}$. 
We introduce the following subsets of the real line
\begin{align*}
	\KK_{N,\alpha}^+ & =\{ [\ab{0},\ab{1}, \dots,\ab{n},\dots]_\alpha^+, \text{ with } (\ab{n})_{n\in\NN}\in\cK_{N}^{\alpha}\} ;\\
	\KK_{N,\alpha}^- & =\{ [\bb{0},\bb{1}, \dots,\bb{n},\dots]_\alpha^-, \text{ with } (\bb{n})_{n\in\NN}\in\cK_{N}^{\alpha}\} .
\end{align*}
One can show that actually $ \KK_{N,\alpha}^\pm \subset \RR^+= [0, \infty)$ (see Lemma~\ref{propertiesmatrices}, Part $(5)$).  We will show in \S~\ref{sec:Cantorproof} that $ \KK_{N,\alpha}^+ $ and $ \KK_{N,\alpha}^-$ are Cantor sets. 
The final ingredient we need for the proof of the existence of the Hall ray is the following Proposition, which 
 will be proved in \S~\ref{sec:Cantorproof}.

\begin{prop}\label{propsumcantorsets}
For any $\alpha\in\cA$ and any $\mu>0$ there exists a positive integer $N$ such that 
the sum of the Cantor sets $\KK^+_{N,\alpha}$ and $\KK^-_{N,\alpha}$ is an interval of size at least $\mu$, that is, if $m^\pm, M^\pm$ denote respectively the minimum and the maximum of  $\KK^\pm_{N,\alpha}$, we have
\[
	\KK^+_{N,\alpha} +\KK^-_{N,\alpha}=  \left[ m^+ + m^-, M^+ + M^- \right] \quad \text{and} \quad M^+ + M^- - (m^+ + m^-)>\mu. 
\]
\end{prop}

\subsection{Concluding arguments for the existence of a Hall ray.}\label{subsec:finalarguments}
Before starting the proof of Thereom~\ref{thm:HallVeech} we want to stress the fact that we will prove the existence of a Hall's ray \emph{relative to} any given cusp,  in the sense that we will show that for any given $\xi$ ideal vertex of  $D$ there is an interval $[r, +\infty] \subset \cL$ whose values can be attained by excursions in the cups associated to $\xi$ (see Remark~\ref{cuspdescription}).

\begin{proofof}{Theorem}{thm:HallVeech}
Fix any $\alpha \in \mathbb{A}$. Let $\alpha {\gamma_1} \dots  {\gamma_{p-1}}$ be the unique right parabolic word starting with $\alpha$ and let  $\mu:=\mu_\alpha$ be given by Corollary~\ref{correnormalizedformula}.
By Proposition~\ref{propsumcantorsets} there exists an integer $\overline{N}$ so that, for $N\geq\overline{N}$, we have that the size of $\KK_{N, \alpha}^+ + \KK_{N,\alpha}^-$ is greater than $\mu$. For brevity, let us write $\cK:= \cK_{N}^\alpha$ and let us denote by $\mathbb{K}^\pm:= \mathbb{K}_{N,\alpha}^\pm$ the corresponding Cantor sets with maxima $M^\pm$ and minima $m^\pm$. We will show that 
\begin{equation}\label{choicer}
[r, + \infty] \subset \cL, \qquad \text{for any }  r>
 \frac{4M_0M_{N+1}}{c(S)^2},
\end{equation}
 where $c(S)$ is the constant in Lemma~\ref{Veech:lemma} and $M_0$, $M_N$ are defined by~\eqref{Mndef}. 
  Fix any real number $L \geq r$ and consider the number $\det(W_{\alpha}) L$. By Proposition~\ref{propsumcantorsets}, since $\mathbb{K}^++\mathbb{K}^-$ contains an interval of length at least $\mu$, we can write $ \det(W_\alpha) L = K \mu + l $ where  $ l  \in \mathbb{K}^++\mathbb{K}^-$. Thus, by definition of $\mathbb{K}^\pm$, there exist $(\ab{n})_{n\in\NN}$ and $(\bb{n})_{n\in\NN}$ in $ \cK$  such that 
\beq \label{sumCantor}
	L= \frac{1}{{\det (W_{\alpha})}}   \left( [\bb{0},\bb{1},\dots,\bb{n}, \dots ]_{\alpha}^- + K \mu 							+[\ab{0},\ab{1},\dots ,\ab{n},\dots]_{\alpha}^+\right).
\eeq

Let us now construct an infinite word $(\cc{n})_{n \in \NN}$ that will give the boundary expansion of an angle $\theta$ such that $L(\theta) = L$. 
 Let us define blocks of entries $C_j$, $j \in \NN$, which we will then concatenate to form the word $(\cc{n})_{n \in \NN}$. Set
\[
	C_{j} = {\bb{j}} \dots {\bb{0}} ( \alpha \gamma_1 \dots \gamma_{p-1})^K \alpha \ab{0} \dots \ab{j} , \qquad j \in \mathbb{N},
\]
where $ (\alpha, \gamma_1, \dots, \gamma_{p-1})^K$ means that the block $\alpha, \gamma_1, \dots, \gamma_{p-1}$ is repeated $K$ times.  Remark that, by definition of the Cantor sets and $\cK$, we have that $\ab{0} \neq \overline{\alpha}$ and $\bb{0} \neq \overline{\alpha}$  and thus $C_k$  satisfy the no-backtracking condition~\eqref{eqnobacktrack}.  Let us choose letters to interpolate between $C_j$ and $C_{j+1}$ as follows. Since the alphabet $\cA$ has cardinality $2d>3$, we can pick  $\delta_j$ such that $\delta_j \neq \overline{\ab{j}}$ and $\ab{j} \delta_j$ is not a cuspidal word and then  $\delta'_j$ such that $\delta'_j \neq \overline{\delta_{j}}$, $\delta'_j \neq \overline{\bb{j+1}}$ and $\delta'_j {\bb{j+1}}$ is not a cuspidal word. Thus, the sequence
\begin{equation}\label{sequencenoN}
	\ab{0} \dots \ab{j} \delta_j \delta_j' \bb{j+1} \dots \bb{0}
\end{equation}
satisfies the no-backtracking condition~\eqref{eqnobacktrack} and, by definition of $\cK$ and choice of $\delta_j, \delta_j'$ does not contain any parabolic word of length $N$. It follows that the infinite word $(\cc{n})_{n \in \NN} $ obtained justapposing the blocks $C_j, \delta_j, \delta_j'$ in increasing order of $j \in \mathbb{N}$ satisfies the no-backtracking condition or, in other words, is actually a boundary expansion of some angle $\theta$.

To show that $L(\theta)=L$, let $(\W{n}{\theta})_{n \in \NN}$  be the wedges associated to $\theta$ and recall that, by Theorem~\ref{boundary_see_largeL}, we have that $	L(\theta) = \limsup_{n \to \infty} 1/{\areaW{\theta}{\W{n}{\theta}}}$,  where $\areaW{\theta}{\W{n}{\theta}}$ is the minimum area with respect to $\theta$ of the vectors in the wedge $\W{n}{\theta}$. Let $(n_k)_{k\in \NN}$ be the sequence of accelerated times. Remark that since  $(\alpha \gamma_1, \dots, \gamma_{p-1})^K \alpha$ is cuspidal, it  is contained in a unique maximal cuspidal word. For each $j \in \mathbb{N}$, let $k_j$ be such that the word $(\alpha,\gamma_1 \dots, \gamma_{p-1})^K \alpha$  in $C_j$ is contained in  the $k_j^\text{th}$ cuspidal word of the decomposition (that is the cuspidal word $\cc{n_{k_j}},\cc{n_{k_j}+1},  \dots, \cc{n_{k_j+1}-1}$). Since we are assuming that $\gamma_1, \dots, \gamma_p$ is right parabolic, the cuspidal word which contains it is right cuspidal and, by Lemma~\ref{lemelementarywords}, we have that $\vv{n}{r}= \vv{n_{k_j}}{r}$ for any $n_{k_j}\leq n < n_{k_j+1}$. Thus, we can compute $\area{\theta}{\vv{n_{k_j}}{r}}$ by evaluating  $\area{\theta}{\vv{n}{r}}$ for $n$ which corresponds to the first occurence of $\alpha$ in the block $(\alpha \gamma_1, \dots, \gamma_{p-1})^K \alpha$ in $C_j$  and applying Corollary~\ref{correnormalizedformula}.  As  $j \to \infty$, recalling the form of the blocks $C_j$ and using  the convergence of $\alpha$-forward and $\alpha$-backward continued fractions  (see Lemma~\ref{lemconvergencecontinuedfraction}), we get 
\begin{equation}\label{limsupseq}
\lim_{j \to \infty }\frac{1}{ \area{\theta}{\vv{n_{k_j}}{r}}} = \frac{1}{\det (W_{\alpha})} \left( [\bb{0},\dots,\bb{n}, \dots ]_{\alpha}^- + K \mu + 
		+[\ab{0},\dots, \ab{n},\dots]_{\alpha}^+\right) = L,
\end{equation}
where the last equality is given  by~\eqref{sumCantor}.
On the other hand, by Proposition~\ref{propcontrolareaselemetarywords} (see~\ref{rightareas}), we have that for $n_{k_j}< n < n_{k_j+1}-1$
\[
\frac{1}{\area{\theta}{\vv{n}{l}}} \leq \frac{4M_0M_{2}}{c(S)^2} \leq  \frac{4M_0M_{N+1}}{c(S)^2}  <r \leq L,
\]
where the last inequalities follow since $M_2\leq M_{N+1}$ by Definition~\eqref{Mndef} of $M_n$ and by choices  of $r$ (see~\eqref{choicer})  and $L$. 
Furthermore, since by construction, for any $j$, the word~\eqref{sequencenoN} does not contain any cuspidal word of length greater than $N-1$, we have that for any $k_j+1 \leq k < k_{j+1}$, $N_k:=n_{k+1}-n_k\leq N-1$. Thus, again by Proposition~\ref{propcontrolareaselemetarywords}, we have that  
\begin{align*}
&\frac{1}{\area{\theta}{\vv{n}{r}}}\leq \frac{4M_0M_{N+1}}{c(S)^2}	<L  & \text{for}\  n_{k_j+1} \leq n < n_{k_{j+1}},\\
& \frac{1}{\area{\theta}{\vv{n}{l}}}\leq \frac{4M_0M_{N+1}}{c(S)^2} <L	& \text{for}\  n_{k_j+1} -1 \leq n \leq n_{k_{j+1}}.
\end{align*}
 This shows that the $\limsup$ which gives $L(\theta)$ is achieved by~\eqref{limsupseq} along the subsequence of times $(n_{k_j})_{j \in \mathbb{N}}$ and hence that  $L(\theta)=L$. 
\end{proofof}

\section{The formula via $\alpha$-backward and forward continued fractions}\label{sec:renormalized_formula}
In this section we prove Theorem~\ref{thmrenormalizedformula}. We first prove some preliminary Lemmas which are used to show the convergence of $\alpha$-backward and $\alpha$-forward continued fractions.

\subsection{Convergence of $\alpha$-backward and $\alpha$-forward continued fractions}
Recall that all matrices denote equivalence classes in $\pslduer$.  
We say that $A$ is \emph{positive} (\emph{strictly positive}) if there exists a representative of $A$ in its equivalence class in $\pslduer$ (that is a matrix $t A$ with $t \neq 0$) with  all entries \emph{positive} (\emph{strictly positive}). As usual, $A^T$ denotes the transpose matrix of $A$. 

\begin{lem}\label{lemproductmatricesA}
Let $\ab{0} \ab{1} \dots \ab{k}$ be any finite word.
The product
\begin{equation}\label{matrixproduct}
	W_{\ab{0}}^{-1} G_{\ab{0}}\dots G_{\ab{k-1}} W_{\ab{k}}
\end{equation}
is positive if and only if $\ab{0} \dots \ab{k}$ satisfies the no-backtracking condition~\eqref{eqnobacktrack}. Furthermore, if $\ab{0} \dots \ab{k}$ satisfies the no-backtracking condition~\eqref{eqnobacktrack}, the matrix product~\eqref{matrixproduct} 
is  strictly positive if and only if $\ab{0} \dots \ab{k}$ is \emph{not} a cuspidal word.
\end{lem}

\begin{proof}
Recall that for any $\alpha \in \cA$, by definition, $W_\alpha$ sends the cone $\RR_+\times\RR_+$ onto the cone $\cW_\alpha$. In order to prove the first statement, consider a word with two letters $\ab{0}\ab{1}$. If $\overline{\ab{0}}\not=\ab{1}$ then $\cW_{\ab{1}}\subset\cU_{\ab{0}}= \RR^2\setminus\cW_{\overline{\ab{0}}}$,
so that $G_{\ab{0}}\cW_{\ab{1}}\subset\cW_{\ab{0}}$
and thus
\[
	W_{\ab{0}}^{-1}G_{\ab{0}} W_{\ab{1}} \big(\RR_+\times\RR_+\big)\subset\RR_+\times\RR_+.
\]
On the other hand $G_{\overline{\ab{0}}}=G_{\ab{0}}^{-1}$ so that $G_{\ab{0}}(\cW_{\overline{\ab{1}}})=\cU_{\ab{0}}$ and thus
\[
	W_{\ab{0}}^{-1}G_{\ab{0}}W_{\overline{\ab{0}}}\big(\RR_+\times\RR_+\big)=\RR_-\times\RR_+.
\]
Therefore the statement (1) is proved for $k=2$. One can easily prove it inductively for any $k$ observing that we have the factorization
\begin{equation}\label{factorization}
	W_{\ab{0}}^{-1}G_{\ab{0}} G_{\ab{1}}\dots G_{\ab{k}}W_{\ab{k+1}}=%
		W_{\ab{0}}^{-1}G_{\ab{0}}\dots G_{\ab{k-1}}W_{\ab{k}}W_{\ab{k}}^{-1}G_{\ab{k}}W_{\ab{k+1}}.
\end{equation}
The second part of the statement is easy to prove and we leave it to the reader.
\end{proof}

Fix a letter $\alpha$ in $\cA$. In the previous section we have introduced two  families of matrices, $A_\beta^\alpha$ and $B_\beta^\alpha$ as $\beta$ varies in $\cA$, see~\eqref{Amatricesdef} and~\eqref{Bmatricesdef}.  The next Lemma  summarizes the properties we need for $\alpha$-forward and  $\alpha$-backward continued fractions, which all follow easily from the previous Lemma.

\begin{lem}\label{propertiesmatrices}
Let $\alpha \ab{0}\dots\ab{k}$ any finite word that satisfies the  no-backtracking condition~\eqref{eqnobacktrack}.
\begin{enumerate}
\item The product  
$A^\alpha_{\ab{0}}A^\alpha_{\ab{1}} \cdots A^\alpha_{\ab{k-1}} W_{\alpha}^{-1} G_\alpha W_\ab{k}$
is positive and it is strictly positive if and only if $\alpha \ab{0}\dots\ab{k}$ is not a cuspidal word.
\item \label{productmatricesB}
The product 
$B^\alpha_{\ab{0}}B^\alpha_{\ab{1}} \cdots B^\alpha_{\ab{k}} R W_\alpha^T (W_{\ab{k}}^T)^{-1} R $
is positive  and it is strictly positive if and only if $\alpha \ab{1}\dots\ab{k}$ is not a cuspidal word.
\item For any parabolic word $\gamma_{0} \gamma_{1}\dots \gamma_{p-1}$ there exists $\mu>0$ such that
	\[
	A^{\gamma_{0}}_{\gamma_{1}} \dots A^{\gamma_{0}}_{\gamma_{p-1}} A^{\gamma_{0}}_{\gamma_{0}} = \begin{pmatrix} 
										1 & \mu \\
										0 & 1
									\end{pmatrix}.
\]
\item  For any $\alpha, \beta \in \cA$ we have
\[
0\leq W_\alpha^{-1} G_\alpha W_\beta * \infty \leq +\infty, \qquad   0\leq R W_\alpha^T  (W_\beta^T)^{-1} R * \infty \leq +\infty.
\]
\item The finite continued fractions $[\ab{0}, \ab{1}, \dots, \ab{k}]_\alpha^\pm \in [0,+\infty]$.
\end{enumerate}
\end{lem}

\begin{proof}
Part $(1)$ follows from Lemma~\ref{lemproductmatricesA} since, by Definition~\eqref{Amatricesdef} of the matrices $A^\alpha_{\alpha_i}= W_\alpha^{-1}G_\alpha G_{\alpha_i} G_\alpha^{-1}W_\alpha $, 
\begin{equation}\label{equalityAmatrices}
A^\alpha_{\ab{0}}A^\alpha_{\ab{1}} \cdots A^\alpha_{\ab{k-1}} W_{\alpha}^{-1} G_\alpha W_\ab{k} =
W_{\alpha}^{-1} G_\alpha G_{\ab{1}}  \dots G_{\ab{k-1}} W_{\ab{k}}. 
\end{equation} 
For Part $(2)$, observe that by Definition~\eqref{Bmatricesdef} of the matrices $B^\alpha_{\alpha_i}=  RW_\alpha^T G_{\alpha_i}^T  (W_\alpha^T)^{-1} R $ we have that
\begin{equation}\label{equalityBmatrices}
B^\alpha_{\ab{0}} \cdots B^\alpha_{\ab{k}} R W_\alpha^T {(W_{\ab{k}}^T)}^{-1} R  =	R W^T_{\alpha} G^T_{\ab{0}}\dots G^T_{\ab{k}} {(W^T_{\ab{k}})}^{-1} R =
		R \big(W_{\ab{k}}^{-1}G_{\ab{k}}\dots G_{\ab{0}} W_{\alpha}\big)^T R.
\end{equation}
Since  being positive is invariant under transposition and conjugation by $R$ and the sequence $\ab{k}\dots \ab{0} \alpha$ satisfies the no-backtracking condition, since $\alpha \ab{0} \dots \ab{k}$ does, also Part (2) follows from Lemma~\ref{lemproductmatricesA}. 
 
To prove Part $(3)$ recall that by Part $(1)$ in Lemma~\ref{parabolic_in_cuspidal} the product $G_{\gamma_0} G_{\gamma_{1}}\dots G_{\gamma_{p-1}}$ is parabolic and fixes the vector $v_{\gamma_0}^r$. Thus, by the equality in~\eqref{equalityAmatrices} we have
	\begin{align*}
		A^{\gamma_0}_{\gamma_{1}}\dots A^{\gamma_{0}}_{\gamma_{p-1}} A^{\gamma_0}_{\gamma_0}   \begin{pmatrix}1\\0 \end{pmatrix} %
			= W_{\gamma_0}^{-1}G_{\gamma_0} G_{\gamma_1}\dots G_{\gamma_{p-1}} W_{\gamma_0} \begin{pmatrix}1\\0 \end{pmatrix} 
			=  W_{\gamma_0}^{-1}G_{\gamma_0} G_{\gamma_1}\dots G_{\gamma_{p-1}} v_{\gamma_0}^r  = W_{\gamma_0}^{-1}v_{\gamma_0}^r= \begin{pmatrix}1\\0 \end{pmatrix} .
	\end{align*}
	Moreover, positivity of $\mu$ follows from Lemma~\ref{lemproductmatricesA}. 

Let us prove Part $(4)$. If $\beta \neq \overline{\alpha}$, the first set of inequalities follows since the product $W_\alpha^{-1} G_\alpha W_\beta$ is positive by Lemma~\ref{lemproductmatricesA} and hence $W_\alpha^{-1} G_\alpha W_\beta * \infty \subset [0,+\infty]$.
Otherwise, since $W_{\overline{\alpha}}* \infty = v_{\overline{\alpha}}^r$ is a generator of the cone $\cU_\alpha$ and hence, since $G_{\alpha } \cU_\alpha = \cW_\alpha$, it  is mapped by $G_{\alpha } $ to a generator of the cone $\cW_\alpha$, we get that $W_\alpha^{-1} G_\alpha W_{\overline{\alpha}} * \infty  \subset \{ 0, \infty\}$. 
To prove the second set of inequalities, remark first that if $\alpha = \beta $ we get trivially the identity matrix which fixes $\infty$. Thus let us assume that $\alpha \neq \beta$.  Since $R^T=R$ and $R$ is an involution in $\pslduer$, 
\begin{equation}\label{transposeinverse}
 R W_\alpha^T  (W_\beta^T)^{-1} R = \left[ \left(  R W_\alpha^{-1} W_\beta R \right)^T\right]^{-1}.
\end{equation}
Recall that by definition the matrix  $W_\beta$  sends $ \RR^\pm \times \RR^\pm $ to the wedges $\pm W_\beta \subset \cW_\beta$ and  since $\alpha \neq \beta$, $\cW_\beta \subset (\RR^2 \setminus \cW_\alpha )  $.  
Thus, 
using also that $R ( \RR^+\times \RR^+)= \RR^+\times \RR^+  $ and  $R ( \RR^\pm\times \RR^\mp)= \RR^\mp \times \RR^\pm  $, we get
\[
 R W_\alpha^{-1} W_\beta  R \ ( \RR^+\times \RR^+ ) \subset  R W_\alpha^{-1} ( \RR^2 \setminus \cW_\alpha ) = \RR^+\times \RR^- \cup \RR^-\times \RR^+.
\]
Thus, $-\infty\leq  R W_\alpha^{-1} W_\beta  R  \ast 0 \leq 0$. Hence, since one can verify that $(A^T)^{-1}\ast z= -1/ A \ast (-1/z) $, it follows  from~\eqref{transposeinverse} that  $0 \leq  R W_\alpha^T  (W_\beta^T)^{-1} R  \ast \infty \leq +\infty$. 

Finally, Part $(5)$ follows immediately from  the Definitions~\eqref{forwardCFdef} and~\eqref{backwardCFdef} of $\alpha$-forward and $\alpha$-backward continued fractions and Part $(1)$ when $k\geq 1$ and Part $(4)$ for   $[ \ab{0}]^\pm_\alpha$.
\end{proof}

We can now prove that the $\alpha$-forward and $\alpha$-backward continued fractions introduced in the previous section are well defined. This follows immediately from the following Lemma. 

\begin{lem}\label{lemconvergencecontinuedfraction}
Assume  that $(\ab{n})_{n\in\NN}$  satisfies the no-backtracking condition and $\alpha_0\neq \alpha \in \cA$.
Then the finite $\alpha$-forward and backward continued fractions $[\ab{0},\ab{1},\dots,\ab{n}]_\alpha^\pm$ converge.
\end{lem}

\begin{proof}
Set
\[
	A^{(n)}:= A^\alpha_{\ab{0}}A^\alpha_{\ab{1}} \cdots A^\alpha_{\ab{n-1}}  W_{\alpha}^{-1} G_\alpha W_\ab{n}, 
	\quad \text{and} \quad
	B^{(n)}:= B^\alpha_{\ab{0}}B^\alpha_{\ab{1}} \cdots B^\alpha_{\ab{n}} R W_\alpha^T (W_{\ab{n}}^T)^{-1} R. 
\]
We will show that both  $\bigcap_{n \in \NN} \big( A^{(n)} * \RR_+ \big)$ and  $\bigcap_{n \in \NN} \big( B^{(n)} * \RR_+ \big)$ are non-empty and consist of a unique point.

By the equalities~\eqref{equalityAmatrices} and~\eqref{equalityBmatrices},  the factorization~\eqref{factorization} and Lemma~\ref{lemproductmatricesA}, we have that
\begin{align*}
A^{(n+1)} * \RR_+ & =  
\left( W_{\alpha}^{-1} G_\alpha G_{\ab{0}}\dots G_{\ab{n-1}} W_{\ab{n}}\right)\left( W_{\ab{n}}^{-1} G_{\ab{n}} W_{\ab{n+1}}\right) * \RR_+  \subset A^{(n)} * \RR_+, \\
 B^{(n+1)} * \RR_+ & = \left( R W^T_{\alpha} G^T_{\ab{0}} \dots  G^T_{\ab{n}} {(W^T_{\ab{n}})}^{-1} R \right)\left(R  W^T_{\ab{n}}       G^T_{\ab{n+1}} {(W^T_{\ab{n+1}})}^{-1} R\right) 
 \subset B^{(n)} * \RR_+.
\end{align*}
Thus, both are intersections of nested sets.  Furthermore, if $v^{(n)}$ and $w^{(n)}$ denote the two column vectors of $A^{(n)}$, we have that
\[
| A^{(n)} * 0 - A^{(n)} * \infty | \leq   \angle (v^{(n)}, w^{(n)}) \leq \frac{c}{|v^{(n)}| |w^{(n)}|}
\]
and a similar estimate holds for $|B^{(n)} * 0 - B^{(n)} * \infty |$.  Thus, to prove that the intersections consists of a unique point it is enough to show that  the norm of $A^{(n)}$ and $B^{(n)}$ are growing.  
If $(\ab{n})_{n\in\NN}$ is eventually cuspidal, one of the column vectors in  $A^{(n)}$ (resp.\ $B^{(n)}$) is eventually constant for $n$ large, but the other grows linearly in norm, thus we are done. On the other hand,  if  $(\ab{n})_{n\in\NN}$ is not eventually cuspidal, by Lemma~\ref{lemproductmatricesA},  $A^{(n)}$ and $B^{(n)}$  factorize as shown above in arbitrarily many  factors which are strictly positive, hence their norms also grow. This concludes the proof. 
\end{proof}

\subsection{The proof of Theorem~\ref{thmrenormalizedformula}.} 
Let $(\ab{n})_{n\in\NN}$ be the cutting sequence associated to $\theta$.
Let $\W{n}{\theta}$ be the sequence of matrices defined by Equation~\eqref{eq:wedgesfromboundaryexpansion}.
We write
\[
	\W{n}{\theta}= \begin{pmatrix}
		\re_\theta(\vv{n}{r} ) & \re_\theta(\vv{n}{l} ) \\
		\im_\theta(\vv{n}{r} ) & \im_\theta(\vv{n}{l} )
		\end{pmatrix}.
\]
The key remark to obtain an expression which splits past and future of the coding is the following.
Since 
$G_{\ab{0}},\dots,G_{\ab{n-1}}$ are in $\slduer$ we have
\[
	\det(W_{\ab{n}})= \det{\W{n}{\theta}}= \re_\theta (\vv{n}{r} ) \im_\theta(\vv{n}{l} )- \re_\theta (\vv{n}{l} ) \im_\theta(\vv{n}{r} )>0.
\]
Moreover $\area{\theta}{\vv{n}{r} }= \re_\theta (\vv{n}{r} )\cdot\im_\theta(\vv{n}{r} )$, and therefore
\begin{equation}\label{eq:inverseareas}
\begin{split}
\frac{1}{\area{\theta}{\vv{n}{r} }} = \frac{1}{\re_\theta(\vv{n}{r} )\cdot\im_\theta(\vv{n}{r} )} 
&=\frac{1}{\det(W_{\ab{n}})} \frac{\re_\theta(\vv{n}{r} )\im_\theta(\vv{n}{l} )- \re_\theta(\vv{n}{l} )\im_\theta(\vv{n}{r} )}{\re_\theta(\vv{n}{r} )\cdot\im_\theta(\vv{n}{r} )}\\
	&=\frac{1}{\det(W_{\ab{n}})}\bigg( \frac{\im_\theta(\vv{n}{l} )}{\im_\theta(\vv{n}{r} )} + \frac{-\re_\theta(\vv{n}{l} )}{\re_\theta(\vv{n}{r} )}\bigg).
\end{split}
\end{equation}
The following Lemma describes how the two terms in the parenthesis transform as $n$ changes. From the previous formula and the Lemma the continued fraction formula will be then easily deduced.

\begin{lem}\label{lem4renormalizedformula}
For any $n$ we have
\begin{equation}\label{eq1lem4renormalizedformula}
	\frac{\im_\theta(\vv{n}{l} )}{\im_\theta(\vv{n}{r} )}=
		B^{\ab{n}}_{\ab{n-1}}\dots B^{\ab{n}}_{\ab{0}} RW^T_{\ab{n}}\big(W^T_{\ab{0}}\big)^{-1}R*
				\frac{\im_\theta(v_0^l)}{\im_\theta(v_0^r)}.
\end{equation}
Moreover for any integer $k>0$ we have
\begin{equation}\label{eq2lem4renormalizedformula}
	\frac{-\re_\theta(\vv{n}{l} )}{\re_\theta(\vv{n}{r} )}=
		A^{\ab{n}}_{\ab{n+1}}\dots A^{\ab{n}}_{\ab{n+k-1}}W_{\ab{n}}^{-1}G_{\ab{n}}W_{\ab{n+k}}*
				\frac{-\re_\theta(v_{n+k}^l)}{\re_\theta(v_{n+k}^r)}.
\end{equation}
\end{lem}
The proof of the Lemma exploits the following linear algebra exercise (Lemma~\ref{lem3renormalizedformula(transfpastfuture)}).   
Consider two matrices $W$ and $W'$ with column vectors respectively $v^r$, $v^l$ and $w^r$, $w^l$, that is
\[
	W=\begin{pmatrix}
		\re(v^r) & \re(v^l) \\
		\im(v^r) & \im(v^l) 
	\end{pmatrix}
	\text{ and }
	W'=\begin{pmatrix}
		\re(w^r) & \re(w^l) \\
		\im(w^r) & \im(w^l)
	\end{pmatrix}.
\]
Given $\pi/2\leq \theta<\pi/2$, recall that $r_\theta$ is the rotation sending the direction $\theta$ onto the vertical direction $\theta=0$. Remark that
\[
	r_\theta\cdot W=\begin{pmatrix}
				\re_\theta(v^r) & \re_\theta(v^l) \\
				\im_\theta(v^r) & \im_\theta(v^l)
			\end{pmatrix}
	\text{ and }
	r_\theta\cdot W'=\begin{pmatrix}
				\re_\theta(w^r) & \re_\theta(w^l) \\
				\im_\theta(w^r) & \im_\theta(w^l)
			\end{pmatrix}.
\]

\begin{lem}\label{lem3renormalizedformula(transfpastfuture)}
Consider a matrix $A$ with $\det(A)>0$ such that $W'=W\cdot A$.
Then we have
\[
	\frac{\im_\theta(w^l)}{\im_\theta(w^r)}= \reflection \cdot A^T \cdot \reflection *
			\frac{\im_\theta(v^l)}{\im_\theta(v^r)}, \qquad \text{and} \qquad  
	\frac{-\re_\theta(v^l)}{\re_\theta(v^r)}= A*\frac{-\re_\theta(w^l)}{\re_\theta(w^r)}.
\]
\end{lem}

\begin{proof}
We obviously have $r_\theta\cdot W'=r_\theta \cdot W\cdot A$.
Observe that  we have
\[
	\begin{pmatrix}
		a & b \\
		c & d
	\end{pmatrix}^{-1}=
	\begin{pmatrix}
		d & -b \\
		-c & a 
	\end{pmatrix},
	\text{ where }
	\begin{pmatrix}
		a & b \\
		c & d
	\end{pmatrix}
	:=A,
\]
where equality hold in $\pslduer$  even if $ad-bc=\det(A)\not=1$.
The Lemma follows directly by a computation starting from  the product
\[
	\begin{pmatrix}
		\re_\theta(v^r) & \re_\theta(v^l) \\
		\im_\theta(v^r) & \im_\theta(v^l)
	\end{pmatrix}
	\cdot
	\begin{pmatrix}
		a & b \\
		c & d
	\end{pmatrix}
	=
	\begin{pmatrix}
	a\re_\theta(v^r)+c\re_\theta(v^l) & b\re_\theta(v^r)+d\re_\theta(v^l) \\
	a\im_\theta(v^r)+c\im_\theta(v^l) & b\im_\theta(v^r)+d\im_\theta(v^l)
	\end{pmatrix}.\qedhere
\]
\end{proof}

\begin{proofof}{Lemma}{lem4renormalizedformula}
For any $n\geq 1$, Equation~\eqref{eq:wedgesfromboundaryexpansion} can be rewritten recursively as $\W{n}{\theta}=\W{n-1}{\theta}H_n$ in terms of the matrix
\[
	H_n:=W_{\ab{n-1}}^{-1}G_{\ab{n-1}}W_{\ab{n}}.
\]
We first prove Equation~\eqref{eq1lem4renormalizedformula}.
Observe that that for any $n\geq2$ we have
\[
	H^T_nH^T_{n-1}= W^T_{\ab{n}}G^T_{\ab{n-1}}G^T_{\ab{n-2}}\big(W^T_{\ab{n-2}}\big)^{-1}.
\]
We have $r_\theta W_i=r_\theta W_{i-1}H_i$ for any $i$ with $1\leq i\leq n$ and moreover $R^2=\Id$, thus Lemma~\ref{lem3renormalizedformula(transfpastfuture)} implies
\[
\begin{split}
	\frac{\im_\theta(\vv{n}{l} )}{\im_\theta(\vv{n}{r} )} &=RH_n^TR*\frac{\im_\theta(v_{n-1}^l)}{\im_\theta(v_{n-1}^r)}
				=RH_n^TR\dots RH_1^TR * \frac{\im_\theta(v_0^l)}{\im_\theta(v_0^r)}\\
				&=RW^T_{\ab{n}}G^T_{\ab{n-1}}\dots G^T_{\ab{1}} G^T_{\ab{0}}\big(W^T_{\ab{0}}\big)^{-1}R*
						\frac{\im_\theta(v_0^l)}{\im_\theta(v_0^r)}\\
				&=B^{\ab{n}}_{\ab{n-1}}\dots B^{\ab{n}}_{\ab{0}} RW^T_{\ab{n}}\big(W^T_{\ab{0}}\big)^{-1}R*
						\frac{\im_\theta(v_0^l)}{\im_\theta(v_0^r)},
\end{split}
\]
where the last inequality follows from the definition~\eqref{Bmatricesdef} of $ B_{\beta}^\alpha= RW_\alpha^T G_{\beta}^T  (W_\alpha^T)^{-1} R $. 

Now we prove Equation~\eqref{eq2lem4renormalizedformula}.
For any integer $k>0$ we have	
\[ 
	r_\theta \W{i+1}{\theta}=r_\theta \W{i}{\theta}H_{i+1}, \qquad \text{for any $i$ with } n\leq i\leq n+k-1, 
\] 
thus Lemma~\ref{lem3renormalizedformula(transfpastfuture)} implies 
\begin{align*}
	\frac{-\re_\theta(\vv{n}{l} )}{\re_\theta(\vv{n}{r} )} = H_{n+1}\dots H_{n+k} * \frac{-\re_\theta(v_{n+k}^l)}{\re_\theta(v_{n+k}^r)}
	&	= W_{\ab{n}}^{-1}G_{\ab{n}}\dots G_{\ab{n+k-1}}W_{\ab{n+k}}* \frac{-\re_\theta(v_{n+k}^l)}{\re_\theta(v_{n+k}^r)}\\
		&= A^{\ab{n}}_{\ab{n+1}} \dots A^{\ab{n}}_{\ab{n+k-1}}W_{\ab{n}}^{-1}G_{\ab{n}}W_{\ab{n+k}}* \frac{-\re_\theta(v_{n+k}^l)}{\re_\theta(v_{n+k}^r)},
\end{align*}
where the last inequality follows from the definition~\eqref{Amatricesdef} of $A_\beta^\alpha:= W_\alpha^{-1}G_\alpha G_\beta G_\alpha^{-1} W_\alpha$. 
\end{proofof}

We have now all elements to conclude the proof of Theorem~\ref{thmrenormalizedformula}.
\begin{proofof}{Theorem}{thmrenormalizedformula} 
By~\eqref{eq:inverseareas} and Lemma~\ref{lem4renormalizedformula}
\begin{multline*}
\frac{1}{\area{\theta}{\vv{n}{r} }}  =\frac{1}{\det(W_{\ab{n}})}\bigg( \frac{\im_\theta(\vv{n}{l} )}{\im_\theta(\vv{n}{r} )} + 
						\frac{-\re_\theta(\vv{n}{l} )}{\re_\theta(\vv{n}{r} )}\bigg) \\
									 =\frac{1}{\det(W_{\ab{n}})} \biggl(B^{\ab{n}}_{\ab{n-1}}\dots B^{\ab{n}}_{\ab{0}}
						RW^T_{\ab{n}}\big(W^T_{\ab{0}}\big)^{-1}R* \frac{\im_\theta(v_0^l)}{\im_\theta(v_0^r)}\\
					+A^{\ab{n}}_{\ab{n+1}} \dots A^{\ab{n}}_{\ab{n+k-1}}W_{\ab{n}}^{-1}G_{\ab{n}}W_{\ab{n+k}}*
					\frac{-\re_\theta(v_{n+k}^l)}{\re_\theta(\vv{n+k}{r})}\biggl).
\end{multline*}
Thus, by Lemma~\ref{lemconvergencecontinuedfraction} and  by Definitions~\eqref{forwardCFdef} and~\eqref{backwardCFdef} of $\ab{n}$-forward and $\ab{n}$-backward continued fractions, 
\begin{align*}
\frac{1}{\area{\theta}{\vv{n}{r} }} & \sim \frac{1}{\det(W_{\ab{n}})} \bigg(B^{\ab{n}}_{\ab{n-1}}\dots B^{\ab{n}}_{\ab{0}}RW^T_{\ab{n}}
						\big(W^T_{\ab{0}}\big)^{-1}R*\infty  						+ A^{\ab{n}}_{\ab{n+1}} \dots A^{\ab{n}}_{\ab{n+k-1}}W_{\ab{n}}^{-1}G_{\ab{n}}W_{\ab{n+k}}*\infty\bigg)\\
					& =\frac{1}{\det(W_{\ab{n}})}\bigg([\ab{n-1},\dots,\ab{0}]_{\ab{n}}^- +
						[\ab{n+1},\ab{n+2},\dots ]_{\ab{n}}^+\bigg).
\end{align*}
We have thus proved the first formula of Theorem~\ref{thmrenormalizedformula}. 
The second one, that is the formula for $\area{\theta}{\vv{n}{l} }$ follows trivially from the first one observing that we have
\[
	\frac{1}{\area{\theta}{\vv{n}{l} }}= \frac{1}{\det(W_{\ab{n}})} \bigg( \frac{\im_\theta(\vv{n}{r} )}{\im_\theta(\vv{n}{l} )}+
							\frac{\re_\theta(\vv{n}{r} )}{-\re_\theta(\vv{n}{l} )}\bigg).
\]
This concludes the proof of Theorem~\ref{thmrenormalizedformula}.
\end{proofof}

\begin{proofof}{Corollary}{correnormalizedformula}
Let $\alpha \gamma_1 \cdots \gamma_{p-1}$ be the right parabolic word starting with $\alpha \in \cA$. By Lemma~\ref{lemproductmatricesA}
there exists $\mu_\alpha>0$ such that 
\[
	A^\alpha_{\gamma_1}\cdots A_{\gamma_{p-1}}^{\alpha} A^\alpha_\alpha = \begin{pmatrix} 1&\mu_\alpha \\ 0&1 \end{pmatrix}.
\]
Thus, for any $n $ such that $\ab{n+1} \cdots \ab{n+p}= \gamma_1 \cdots \gamma_{p-1}\alpha$, we have that
\begin{align*}
[\ab{n+1},\ab{n+2}, \dots 
 ]_{\alpha}^+ & =  [\gamma_{1}, \dots, \gamma_{p-1}, \alpha, \ab{n+p+1},\ab{n+p+2}, \dots]_{\alpha}^+  \\  &
 =	A^\alpha_{\gamma_1}\cdots A_{\gamma_{p-1}}^{\alpha} A^\alpha_\alpha * [\ab{n+p+1},\ab{n+p+2},\dots]_{\alpha}^+ 
 = \mu_\alpha +  [\ab{n+p+1},\ab{n+p+2},\dots]_{\alpha}^+ .
\end{align*} 
Hence, recalling that $\ab{n}=\alpha$,  the Corollary follows immediately by  Theorem~\ref{thmrenormalizedformula}. 
\end{proofof}

\section{Cantor sets}\label{sec:Cantorproof}
In this section we explain with more details the construction of the Cantor sets introduced in \S~\ref{sec:Hall} and we prove the technical Proposition~\ref{propsumcantorsets}. 

\subsection{Hall's gap condition}

Let $\KK$ be any Cantor set in $\RR$. A \emph{slow subdivision} of $\KK$ is a family of closed sets $\big(\KK(n)\big)_{n\in\NN}$  satisfying the following properties.
\begin{enumerate}
\item	Any set $\KK(n)$ is the union of disjoint closed intervals.
\item	For any $n$ there is exactly one compact interval $K$ in $\KK(n)$ and a non-empty open subinterval $B_K$ of $K$ such that
	\[
		I\cap\KK(n+1)=K\setminus B_K=K^L \sqcup K^R,
	\]
where $K^L$ and $K^R$  are two disjoint, non-empty, closed subintervals in $\KK(n+1)$.
In particular $\KK(n+1)\subset\KK(n)$ and $\KK(n+1)$ is obtained removing the interval $B_K$ from $K$. 
\item We have 
	\[
		\bigcap_{n\in\NN}\KK(n)=\KK.
	\]
\end{enumerate}
Remark that if $ B_1, \dots, B_n, \dots$ is any enumeration of the holes of $\KK$, setting $\KK(0) = [\min \KK , \max \KK] $ $\KK(n+1)$ and $\KK(n+1)=\KK(n) \setminus B_{n+1}$ we obtain a slow subdivision of $\KK$.   

We say that a slow subdivision $\big(\KK(n)\big)_{n\in\NN}$  of the Cantor set $\KK$ satisfies the \emph{gap condition} if for any $n$, the interval $B_K \subset K \in \KK(n)$ such that $\KK(n+1)\cap K = K \setminus B_K$ (see $(2)$ in the definition above) satisfies
\[
	|B_K|<|K^L| \quad \text{ and } \quad |B_K|<|K^R|.
\]
We say that the Cantor set $\KK$ satisfies the \emph{gap condition}  if it admits a slow subdivision which satisfies the {gap condition}.
We call \emph{holes} of a Cantor set $\KK$ the connected components of the complement which are contained in the interval $[\min \KK, \max \KK]$.
Remark that holes are maximal open intervals in the complement.  
 
Given two Cantor sets $\KK$ and $\FF$, we say that the pair of Cantor sets $(\KK,\FF)$ satisfies the \emph{size condition} if the length $|\KK|$ of $\KK$ is bigger than the length of any hole in $\FF$ and vice versa the length $|\FF|$ of $\FF$ is bigger than the length of any hole in $\KK$.

We have the following

\begin{thm}[Hall]\label{thm:Hall}
Let $\KK$ and $\FF$ be two Cantor sets in $\RR$, each one satisfying the gap condition.
Assume that the pair $(\KK,\FF)$ satisfies the size condition.
Then we have
\[
	\KK+\FF=\bigl[\min(\KK)+\min(\FF), \max(\KK)+\max(\FF)\bigr].
\] 
\end{thm} 
This result is a slight reformulation\footnote{The theorem proved by Hall in~\cite[pp.~968--970]{Hall} is stated for ternary Cantor sets. Hall assumes a condition that is equivalent to the size condition for ternary Cantor sets.} of the Theorem proved in~\cite[pp.~968--970]{Hall} which is a key part in the original proof of the existence of the Hall ray for the classical spectrum. 
For the convenience of the reader, we include a proof of this  Theorem in the Appendix~\ref{sec:thmHall}.

In the  next sections we will show that for $N$ sufficiently large the Cantor sets we consider satisfy the gap and size conditions and hence we can apply Theorem~\ref{thm:Hall} to prove Proposition~\ref{propsumcantorsets}.

\subsection{Description of the Cantor sets holes}\label{subsec:Cantordescription}
Throughout the section, we fix a letter $\alpha\in\cA$ and a positive integer $N$. To simplify the notation, we denote by  $\cK:= \cK^\alpha_N$ and by  $\KK^\pm$  the  Cantor sets $\KK^\pm_{N,\alpha}$.  Similarly, we will denote by $A_\beta$ (resp.\ $B_\beta$) the matrices $A^\alpha_\beta$ (resp.\ $B^\alpha_\beta$) introduced in \S~\ref{sec:Hall}, see (\ref{Amatricesdef},~\ref{Bmatricesdef}), dropping the explicit dependence on $\alpha$. 
For any word $\ab{0} \dots \ab{n}$ which satisfies the no-backtracking condition~\eqref{eqnobacktrack}, let us define 
\[
 I_\alpha^\pm (\ab{0} , \ab{1}, \dots ,\ab{n})  = \{ [\bb{0}, \bb{1}, \dots, \bb{n} , \dots ]^\pm_\alpha \text{ s.t.  $(\bb{n})_{n \in \NN}$ satisfies~\eqref{eqnobacktrack} and $\bb{i} =\ab{i}$ for $0 \leq i \leq n$} \}. 
\]
One can see that $I_\alpha^\pm (\ab{1}, \dots, \ab{n})$ are closed intervals in $\mathbb{R}$ and that for fixed $n$ they have disjoint interiors and cover $\mathbb{R}$.
Furthermore, it follows from Part $(5)$ of Lemma~\ref{propertiesmatrices} that
\begin{equation}\label{positivedecomp}
[0,+\infty] = \bigcup_{\beta\neq \overline{\alpha}} I^+_\alpha ( \beta) , \qquad [0,+\infty] =  \bigcup_{\beta\neq \overline{\alpha}} I^-_\alpha (\beta).
\end{equation}
It hence follows from their definition that the sets $\KK^\pm$ are obtained by removing from $[0,+\infty]$ all intervals of the form
\[
	 I_\alpha^\pm (\ab{1}, \dots, \ab{n-N}, \beta_1, \dots, \beta_N), \text{ with $\beta_1 \dots \beta_N$ cuspidal word.}
\]
We can also assume that $\ab{1} \dots \ab{n-N}$ does not contain any cuspidal word of length $N$.
We will call intervals of this form \emph{deleted intervals of level n}.
Notice that two such intervals cannot intersect in their interior, but deleted intervals of different levels can have common endpoints.
More precisely, any deleted interval of level $n$ has a common endpoint with a deleted interval of level $n+N$.
Thus, to describe the Cantor set structure of $\mathbb{K}^\pm$, we are now going to \emph{group} deleted intervals to describe \emph{holes} of the Cantor set.

We will give definitions for $\mathbb{K}^+$ and $\mathbb{K}^-$ in parallel. By Lemma~\ref{propertiesmatrices}, let 
\[
x_0^+:=0 < x_1^+ < \dots < x^+_{2d-2}< x^+_{2d-1}:= +\infty
\]
be the endpoints of the intervals $ I^+_{\alpha}(\beta)$, $\beta \neq \overline{\alpha}$, arranged in increasing order. Similarly let 
 \[ x_0^-:=0 < x_1^- < \dots < x^-_{2d-2}< x^-_{2d-1}:= +\infty
\]
be the endpoints of the intervals $I^-_{\alpha}(\beta)$, $\beta \neq \overline{\alpha}$ also in increasing order.
By convention, we also set $x_{2d-1}^\pm=+ \infty = -\infty$. 

\begin{notation}\label{rholambda} For each $0\leq i \leq  2d-1$, let $\lambda^-_i$ and $\rho^-_i$ denote the letters such that $x_i^-$ is a right endpoint of  $I^-_{\alpha}(\lambda^-_i)$ and a left endpoint of  $I^-_{\alpha}(\rho^-_i)$.  Similarly let $\lambda^+_i$ and $\rho^+_i$ denote the letters such that $x_i^+$ is a right endpoint of  $I^+_{\alpha}(\lambda^+_i)$ and a left endpoint of  $I^+_{\alpha}( \rho^+_i)$, see Figure~\ref{fig:cantor}. 
\end{notation}

\begin{figure}
\centering
\def\svgwidth{0.9\textwidth}
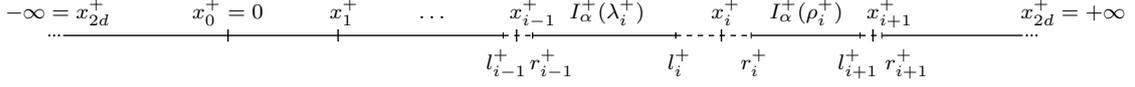
\caption{A schematic representation of the Cantor set $\KK^+$. The dashed segments are some holes of level zero.}
\label{fig:cantor}
\end{figure}


For  any $x_i^\pm$ with $0\leq i \leq  2d-1$, let us group the deleted intervals around $x_i^\pm$ as follows. Let  $ I_\alpha^\pm ( \beta_1, \dots , \beta_N)$  be the unique deleted interval of length $N$ which has $x_i^\pm$ as left endpoint, namely assume $ \beta_1 \dots  \beta_N$ is the unique left cuspidal word that start with $\beta_1:= \lambda_i^\pm$ (see Remark~\ref{unique}). Consider the right endpoint of $ I_\alpha^\pm ( \beta_1, \dots, \beta_N)$. There exists a unique deleted interval of length $2N$ which has this point as its left endpoint. Continuing in this way, define at step $k$ a deleted interval of length $kN$ whose left endpoint coincide with the right endpoint of the interval at step $k-1$. Since the size of these intervals shrink exponentially in $k$, there exists a point $r_i^\pm$ defined as limit point of the right endpoints of these intervals as $k$ grows. Repeating the construction on the other side, we can defined similarly the points $l_i^\pm$. Remark that for $0\leq  i \leq 2d-2$ we have that $l_i^\pm<x_i^\pm<r_i^\pm$ and by construction 
  the intervals $(l_i^\pm, r_i^\pm)$ for $0\leq i \leq 2d-2$ and $(-\infty, r_{2d-1})$ and $(l_{2d-1}, +\infty)$ are union of deleted intervals. 

\begin{defi}
The \emph{holes of first generation} for $\KK^\pm$ are the intervals $(l_i^\pm, r_i^\pm)$  for $ 0<i < 2d-1$. 
The \emph{ holes}
\emph{of generation} $k$ \emph{for} $\KK^+$, whose union will be denoted $H^+_k$,  consist of all intervals of the form
\[
\left( A_{\ab{1}} \dots A_{\ab{k}}* l_i^+, A_{\ab{1}} \dots A_{\ab{k}} * r_i^+  \right), \qquad \lambda_i^+ \neq \overline{\ab{k}}, \quad  \rho_i^+ \neq \overline{\ab{k}},
\]
where ${\ab{1}} \dots {\ab{k}}$ are the first letters of a sequence $(\ab{n})_{n \in \NN} \in \cK$.
The \emph{holes}
 \emph{of generation $k$ for $\KK^-$}, whose union will be denoted $H^-_k$, consist of all intervals of the form
\[
\left( B_{\ab{1}} \dots B_{\ab{k}} * l_i^-  ,  B_{\ab{1}} \dots B_{\ab{k}} * r_i^-  \right), \qquad \lambda_i^- \neq \overline{\ab{k}}, \quad  \rho_i^- \neq \overline{\ab{k}},
\]
where ${\ab{1}} \dots {\ab{k}}$ are the first letters of a sequence $(\ab{n})_{n \in \NN} \in \cK$.
\end{defi}

\begin{lem}\label{lemdescriptionholes}
The mimima and maxima $m^\pm$ and $M^\pm$ of $ \mathbb{K}^\pm$ are given by $m^\pm = r_0^\pm $ and  $M^\pm = l_{2d-1}^\pm $. The Cantor sets $\mathbb{K}^\pm$ are obtained removing from $[m^\pm, M^\pm]$ the union over $k$ of all holes of generation $k$, that is
\[
	\mathbb{K}^\pm = [m^\pm , M^\pm] \setminus \bigcup_k H^\pm_k = [r_0^\pm , l_{2d-1}^\pm] \setminus \bigcup_k H^\pm_k.
\]
\end{lem}
\begin{proof} 
Let us first show that $\mathbb{K}^\pm$ are closed. If $y^\pm$ belongs to the complement of $\mathbb{K}^\pm$, it is described by a word $(\ab{n})_{n \in \NN} \in \cK$ which contains a parabolic word of length $N $ and thus 
$y^\pm$ belongs to some deleted interval for $\mathbb{K}^\pm$, call it $I$. 
Remark that the endpoints of $I$ are described by eventually cuspidal words. Thus, either $(\ab{n})_{n \in \NN}$  is not eventually cuspidal and  hence $y^\pm$ belongs to the interior of $I$, or $(\ab{n})_{n \in \NN}$  is eventually cuspidal and it has another  eventually cuspidal expansions $(\bb{n})_{n \in \NN}$. 
  In this case, for $n$ large enough,  $y^\pm$ is the common endpoint of the two adjacent deleted intervals  $I(\ab{1}, \dots, \ab{n})$ and $I(\bb{1}, \dots, \bb{n})$. 
  In both cases, $y^\pm$ is contained in an open interval in the complement of   $\mathbb{K}^\pm$ and thus $\mathbb{K}^\pm$ is closed.  

Let us denote by $\partial H^\pm_k$ the endpoints of holes of generations $k$ for  $\KK^\pm$. We will show first that
\begin{equation}\label{inclusions}
	\overline{\bigcup_k  \partial H^\pm_k} \subseteq  \mathbb{K}^\pm \subseteq [m^\pm, M^\pm] \setminus \bigcup_k H^\pm_k  
\end{equation}
and then that the inclusions are indeed equalities. 
It follows from Lemma~\ref{propertiesmatrices} and the definitions of  $\mathbb{K}^\pm$ and $\alpha$-continued fractions that  $\mathbb{K}^\pm \subset [0,+\infty]$.   Let us now show that $r_i^\pm$ for $0\leq i < 2d-1$ and $ l_i^{\pm}$ for $0< i \leq 2d-1$ belong to $\KK^\pm$. This  implies in particular that $r_0^\pm$ and $l_{2d-1}^\pm$ are respectively the minimum and maximum of  $\KK^\pm$, since by construction $[0, r_0^{\pm})$ and $(l_{2d-1}^\pm, +\infty]$ are union of deleted intervals and hence do not intersect $\KK^\pm$.  
If, by contradiction, $r_i^+$ does not belong to $\KK^+$, it is contained in the interior of a deleted interval $I_a$ or it is the common endpoint of two deleted intervals $I_b$ and $I_c$ respectively of levels $n$ and $n+N$, for some $n$.
In both cases, since $r_i^{+}$ is the limit point of the right endpoints of deleted intervals of shrinking size, one can find a deleted interval which is strictly contained in either $I_a$, $I_b$ or $I_c$.
Since deleted intervals cannot intersect in their interior, we have a contradiction. 
Similarly, one can show that all the endpoints of holes of generations $k$ for  $\KK^\pm$  belong to  $\KK^\pm$. Indeed, remark first that if $I^\pm:=I^\pm(\bb{1}, \dots, \bb{n})$, 
\begin{equation}\label{transformedI}
 A_{\ab{1}} \dots A_{\ab{k}} * I^+ = I^+(\ab{1}, \dots ,{\ab{k}} ,  \bb{1}, \dots, \bb{n}), \qquad B_{\ab{1}} \dots B_{\ab{k}} * I^- = I^-(\ab{1}, \dots ,{\ab{k}} ,  \bb{1}, \dots, \bb{n}).
\end{equation} 
Thus, consider  for example the hole endpoint $ A_{\ab{1}} \dots A_{\ab{k}} * l_i^+$ (the others are treated similarly).
If it did not belong to $\KK^+$, it would belong to the interior of a deleted interval  of the form $I^+(\ab{1} , \dots, \ab{k}, \bb{1}, \dots, \bb{n})$ or it would be the common endpoint of two intervals of a similar form of two different levels $n$ and $n+N$ for some $n$.
In the first case, by~\eqref{transformedI}, $l_i^+$ would belong to the interior of $ I^+(  \bb{1}, \dots, \bb{n})$, which is also a deleted interval and this, as we proved above, gives a contradiction.
In the second case, Equation~\eqref{transformedI} would similarly imply that we can find a deleted interval which has $l_i^+$ as an endpoint and contains a deleted interval of smaller size.
This proves that $\cup_k \partial H_k^\pm\subset\KK^+$.
Since $\KK^+$ is closed, we have shown the first inclusion in~\eqref{inclusions}. 

To prove the second inclusion in~\eqref{inclusions}, remark that the holes in $H^+_k$  (resp.\ $H_k^-$) are by construction union of intervals of the form $A_{\ab{1}} \dots A_{\ab{k}} * I^+$  (resp.\  $ B_{\ab{1}} \dots B_{\ab{k}} * I^-$) where $I^\pm$ are deleted intervals, which, by~\eqref{transformedI} are again deleted intervals.
Thus, $H^+_k$ is included in the complement of $\KK^\pm$ and hence $\mathbb{K}^\pm \subset [m^\pm, M^\pm] \setminus \cup_k H^\pm_k $.

Remark also that  $[m^\pm, M^\pm] \setminus \cup_k H^\pm_k$ is obtained removing open intervals and hence is closed.
We proved at the beginning that $\mathbb{K}^\pm$ is also closed. Thus, to conclude it is enough to show that $\cup_k  \partial H^\pm_k$ is  dense in $ [m^\pm, M^\pm] \setminus \cup_k H^\pm_k$, since taking its  closure this forces the  inclusions in~\eqref{inclusions} to be equalities.
A point in $ [m^+, M^+] \setminus \cup_{k=0}^n H^+_k$ (the treatment of $H^-_k$ is analogous) belongs to an intersection in $k$ of  complementary intervals to $H^+_k$, which are intervals of the form 
\begin{equation}\label{complI}
(  A_{\ab{1}}\dots A_{\ab{k}}
\ast
r_{i}^+,\  A_{\ab{1}}\dots A_{\ab{k}}
\ast l_{i+1}^+\big) , 
 \qquad  \rho_i^\pm \neq \overline{\ab{k}}, \quad  \lambda_{i+1}^+ \neq \overline{\ab{k}}.
\end{equation}
Since, by Notation~\ref{rholambda},  $r_i^+ \in I^+(\rho_i^+)$ and $l_{i+1}^+ \in I^+(\lambda_{i+1}^+)$, the assumption on $\rho_i^+ $ and $ \lambda_{i+1}^+$,~\eqref{transformedI} and~\eqref{positivedecomp} imply that $A_{\ab{k}} \ast r_i^+$ and $A_{\ab{k}} \ast l_{i+1}^+$ both belong to $[0, +\infty]$. Thus, the interval~\eqref{complI} is contained in  $ A_{\ab{1}}\dots A_{\ab{k-1}} \ast [0,+\infty]$, whose size shrinks to zero as $k$ grows by Lemma~\ref{lemconvergencecontinuedfraction}.  Density of endpoints follows. 
\end{proof}

\subsection{Verification of the gap condition and end of the proof}
In this section, we give the proof of Proposition~\ref{propsumcantorsets}. 
In order to verify the gap condition for the Cantor sets $\mathbb{K}^\pm$, we will use the 
 following general estimate on the distortion of distances under a M\"obius transformation.

\begin{lem}\label{lemdistortionhomography}
Consider a M\"obius transformation $g(t)=(at+b)/(ct+d)$ with $ad-bc=1$ and fix two real numbers $x<y$ such that the pole $g^{-1}(\infty)=-d/c$ of $g$ does not belong to the closed interval $[x,y]$.
For any $t$ with $x<t<y$ the estimates below hold.
\begin{enumerate}
\item If $g^{-1}(\infty)<x$ then we have
\[
	\frac{\abs{x-t}}{\abs{y-t}}<\frac{\abs{g(x)-g(t)}}{\abs{g(y)-g(t)}}<\frac{\abs{y-g^{-1}(\infty)}^2}{\abs{x-g^{-1}(\infty)}^2}\frac{\abs{x-t}}{\abs{y-t}}.
\]
\item If $y<g^{-1}(\infty)$ then we have
\[
	\frac{\abs{y-g^{-1}(\infty)}^2}{\abs{x-g^{-1}(\infty)}^2}\frac{\abs{x-t}}{\abs{y-t}}<\frac{\abs{g(x)-g(t)}}{\abs{g(y)-g(t)}}<\frac{\abs{x-t}}{\abs{y-t}}.
\]
\end{enumerate}
\end{lem}

\begin{proof}
We just prove the first case, the second being the same.
Recall that $\abs{g'(t)}=\abs{ct+d}^{-2}$ and that this is a decreasing function on $[x,y]$, since the pole of $g$ satisfies $g^{-1}(\infty)<x$.
We have
\begin{align*}
&	 \abs{g'(t)}\cdot\abs{t-x}<\abs{g(t)-g(x)}<\abs{g'(x)}\cdot\abs{t-x},\\
&	\abs{g'(t)}\cdot\abs{t-y}>\abs{g(t)-g(y)}>\abs{g'(y)}\cdot\abs{t-y}.
\end{align*}
Thus the Lemma follows recalling that for any pair of points $r$ and $s$ we have
\[
	\frac{\abs{g'(r)}}{\abs{g'(s)}}=\frac{(cs+d)^2}{(cr+d)^2}=\frac{\big(s-(-d/c)\big)^2}{\big(r-(-d/c)\big)^2}=
						\frac{\big(s-g^{-1}(\infty)\big)^2}{\big(r-g^{-1}(\infty)\big)^2}.\qedhere
\]
\end{proof}

The key ingredient to verify the gap condition for $\mathbb{K}^+ $ is the following lemma, in which the length of  a hole in $H_k^+$ is compared with the lengths of the adjacent intervals in the complement of $H_k^+$.

\begin{lem}\label{lemgapconditionfuture}
If $N $ is sufficiently large, 
for any sequence  $(\ab{n})_{n \in \NN}$ in $\cK= \cK_{N}^{\alpha}$,    
 any $k$ and any $0\leq i \leq 2d-1$  such that both $ \lambda^+_i \not=\overline{\ab{k}}$ and $\rho^+_i \not=\overline{\ab{k}}$, the following two conditions are satisfied:
\begin{align*}
	|\big(  A_{\ab{1}}\dots A_{\ab{k}} *l_{i}^+,  A_{\ab{1}}\dots A_{\ab{k}} * r_{i}^+\big) | &< %
	|\big(  A_{\ab{1}}\dots A_{\ab{k}}*r_{i-1}^+,  A_{\ab{1}}\dots A_{\ab{k}}* l_{i}^+\big) |,\\
	|\big(  A_{\ab{1}}\dots A_{\ab{k}}*l_{i}^+,  A_{\ab{1}}\dots A_{\ab{k}}* r_{i}^+\big) |&< %
	|\big(  A_{\ab{1}}\dots A_{\ab{k}}*r_{i}^+, A_{\ab{1}}\dots A_{\ab{k}}*l_{i+1}^+\big) |,
\end{align*}
where the index $i$ in $r_i^+,l_i^+$ should be considered modulo $2d$. 
\end{lem}
 
\begin{proof}
Let $z\mapsto g(z) $ denote the M\"obius transformation given by the matrix  
$A_{\ab{1}}\cdots A_{\ab{k}}$. It is convenient to change coordinates to reduce to a standard interval as follows.  For any $0\leq i \leq 2d-1$,   let $\psi_i$ be the unique M\"obius transformation that sends $x_{i-1}$, $x_{i}$ and $x_{i+1}$ respectively to $-1$, $0$ and $1$. 
  Consider the M\"obius transformation $g_i: = g \circ \psi_i^{-1}$ and remark that 
$ g(z) = {g}_i(\psi_i(z) ) $. Thus, if  we set
\[
l_i':= \psi_i(l_i^+), \quad r_i':=\psi_i(r_i^+),  \qquad \text{for} \quad 0\leq i \leq 2d-1,
\]
 we equivalently want to show that
\begin{equation}\label{toprove}
\frac{|( g_i(l_{i}')-  g_i( r_{i}') ) |}{ | g( r_{i-1}') - g_i (l_{i}')) |}
 = \frac{|(  g(l_{i}^+)-  g( r_{i}^+) ) |}{ | g(r_{i-1}^+) - g
(l_{i}^+)   |}<1, 
\qquad 
\frac{|( g_i(l_{i}')-  g_i( r_{i}') ) |}{ |  g_i (l_{i+1}')-  g( r_{i}') ) |}
 = \frac{|(  g(l_{i}^+)-  g( r_{i}^+) ) |}{ |  g
(l_{i+1}^+)-  g(r_{i}^+) |} <1.
\end{equation}
Remark that for any $0\leq i \leq 2d-1$ we have
\begin{equation}\label{ordering}
-1 < r'_{i-1}<l'_{i} <0 <r'_{i} < l'_{i+1}< 1.
\end{equation}
Recall that for any $0\leq i \leq 2d-1 $ the points $l_i^+=l_i^+(N)$ and $r_i^+=r_i^+(N)$  defined in  \S~\ref{subsec:Cantordescription}  converge to $x_i$ as $N$ grows. Thus, by continuity of $g_i$, it follows that   for any $0\leq i \leq 2d-1 $ 
\begin{equation}\label{convergenceinN}
\lim_{N\to \infty } r_{i-1}'(N)=-1, \quad \lim_{N\to \infty } r_{i}'(N)=\lim_{N\to \infty } l_{i}'(N)=0, \quad \lim_{N\to \infty }l_{i+1}'(N)=1. 
\end{equation}
Consider ${N}=2$ and set $\underline{l}_i':=l_i'(2)$ and $\underline{r}_i':=r_i(2)'$.
Fix  a constant $C>1$ such that 
\begin{equation*}
\sup_{0\leq i \leq 2d-2} \frac{4} {|\underline{l}_{i}+1|^2} <C \quad \text{and} \quad  \sup_{0\leq i \leq 2d-2} \frac{4}
{|1-\underline{r}_{i}|^2}< C.
\end{equation*}
 We can then choose $N> 1$ large enough so that we have both
\begin{equation}\label{smallholes}
\sup_{0\leq i \leq 2d-2} \frac
{|l_i'-r_i'|}
{|l_{i+1}'-r_i'|}
<
\frac{1}{C} \quad 
 \textrm{and} \quad 
\sup_{0\leq i \leq 2d-2} \frac
{|l_i'-r_i'|}
{|l_{i}'-r_{i-1}'|}
<
\frac{1}{C},
\end{equation}
which is possible since by~\eqref{convergenceinN}, as $N$ grows,   $l_{i}'- r_i'$ tends to $0$, while the denominators both tend to $1$. 
 
Observe now  that, since $I^+(\lambda^+_i)=[x_{i-1},x_i]$ and  $I^+(\rho^+_i)=[x_{i},x_{i+1}]$,  by Notation~\ref{rholambda},  the assumptions $\lambda^+_i \not=\overline{\ab{k}} $ and $\rho^+_i \not=\overline{\ab{k}}$  guarantee  (by~\eqref{transformedI} and Lemma~\ref{propertiesmatrices}) that the image of the interval $(x_{i-1}, x_{i+1})$ under $g$ is contained in $\RR_+$ and thus that the pole $g^{-1}(\infty)$ does not belong to the interior of the interval $(x_{i-1}, x_{i+1})$  whereas it may be one of the endpoints. It follows that for any $0\leq i \leq 2d-1$ the closed interval $[r_{i-1}',l_{i+1}'] \subset (-1,1)=\psi_i\left((x_{i-1},x_{i+1})\right)$ does not contain the pole $g_i^{-1}(\infty)$.  

If $g_i^{-1}(\infty)< -1< r_{i-1}'$,  
Lemma~\ref{lemdistortionhomography} (applied  to $x:=r'_{i-1}, t:= l'_i, y:=r'_i$) and~\eqref{smallholes} imply that
\begin{equation}\label{firstcase}
\frac
{|g_i(l'_i)-g_i(r'_i)|}
{|g_i(r'_{i-1})-g_i(l'_i)|}
<
\frac
{|l'_i-r'_i|}
{|r'_{i-1}-l'_i|}
<
\frac{1}{C}<1.
\end{equation}
Again Lemma~\ref{lemdistortionhomography}, applied this time to  $x:=l'_{i}, t:= r'_i, y:=l'_{i+1}$, also gives  that
\begin{equation}\label{secondcase}
\frac
{|g_i(l'_i)-g_i(r'_i)|}
{|g_i(l'_{i+1})-g_i(r'_i)|}
< \frac
{|l'_{i+1}-g_i^{-1}(\infty)|^2}
{|l'_{i}-g_i^{-1}(\infty)|^2}
\frac
{|l'_i-r'_i|}
{|l'_{i+1}-r'_i|} .
\end{equation}
Using that $g_i^{-1}(\infty)< -1$ and the function $z \mapsto |l'_{i+1}-z|/|l'_i-z|$ is monotonically increasing for $z<l'_i$ and that by~\eqref{ordering} we have   $l'_{i+1}< 1$, while $\underline{l}'_{i} \leq l'_i$ since $N \geq 2$, we have 
\begin{equation*}
\frac
{|l'_{i+1}-g_i^{-1}(\infty)|^2}
{|l'_{i}-g_i^{-1}(\infty)|^2} \leq \frac{|1-(-1)|^2}{|\underline{l}'_{i}-(-1)|^2} < 
\frac{4} {|\underline{l}_{i}+1|^2} <C.
\end{equation*}
This, together with~\eqref{secondcase} and~\eqref{smallholes}, concludes with~\eqref{firstcase} the proof of the two inequalities in~\eqref{toprove}. 
   
If $l_{i+1}'<1 <g_i^{-1}(\infty) $,  reasoning in a similar way, Lemma~\ref{lemdistortionhomography} and~\eqref{smallholes} imply that 
\[ \frac
{|g_i(l'_i)-g_i(r'_i)|}
{|g_i(l'_{i+1})-g_i(r'_i)|}<
\frac
{|l'_i-r'_i|}
{|l'_{i+1}-r'_i|}
<
\frac{1}{C} , \] 
and, using that this time $z\mapsto (z-r'_{i-1})^2/ (z-r'_i)^2$ is increasing for $z>r'_i$,  $r'_{i-1}< x_{i-1}$ and that  $\underline{r}_{i}\leq r'_{i}$ since $N \geq 2$, also that 
\[ \frac
{|g_i(l'_i)-g_i(r'_i)|}
{|g_i(r'_{i-1})-g_i(l'_i)|}
< \frac
{|r'_{i-1}-g_i^{-1}(\infty)|^2}
{|r'_{i}-g_i^{-1}(\infty)|^2}
\frac
{|l'_i-r'_i|}
{|r'_{i-1}-l'_i|} < 
\frac{4}
{|1-\underline{r}'_{i}|^2}
\frac{|l'_i-r'_i|}
{|r'_{i-1}-l'_i|}
<1 . \] 
This concludes  the proof. 
\end{proof}
An analogous Lemma, whose proof we leave to the reader, also holds for the Cantor set $\KK^-$.  

\begin{lem}\label{lemgapconditionpast}
If $N$ is sufficiently large, let  $(\ab{n})_{n \in \NN}$ be any sequence in $\cK= \cK_{N}^{\alpha}$.   
Then for any $k$ and any $0\leq i \leq 2d-1$  such that both $\lambda^-_i \not=\overline{\ab{k}}$ and $\rho^-_i \not=\overline{\ab{k}}$ the following two conditions are satisfied:
\begin{align*}
	|\big(  B_{\ab{1}}\dots B_{\ab{k}}*l_{i}^-,  B_{\ab{1}}\dots B_{\ab{k}}* r_{i}^-\big) | &< %
	|\big(  B_{\ab{1}}\dots B_{\ab{k}}*r_{i}^-,  B_{\ab{1}}\dots B_{\ab{k}}* l_{i+1}^-\big) |,\\
	 |\big(  B_{\ab{1}}\dots B_{\ab{k}}*l_{i}^-,  B_{\ab{1}}\dots B_{\ab{k}}* r_{i}^-\big) | &< %
	|\big(  B_{\ab{1}}\dots B_{\ab{k}}*r_{i-1}^-,  B_{\ab{1}}\dots B_{\ab{k}}* l_{i}^-\big) |,
\end{align*}
where the index $i$ in $r_i^-,l_i^-$ should be considered modulo $2d$. 
\end{lem}

\begin{proofof}{Proposition}{propsumcantorsets} 
Remark that as $N$ tends to infinity, the minima $m_N^\pm $ and maxima $M_N^\pm$ of  $\KK_{\alpha, N}^{\pm}$ (which by Lemma~\ref{lemdescriptionholes} are given by $r_0^\pm$ and $l_{2d-1}^\pm$) tend respectively to $0$ and $+\infty$. Thus, the size of $\KK^\pm$ increases as $N$ grows. Furthermore, the holes in $\KK^\pm$ shrink exponentially. Hence, we can choose $N $ large enough such that (1)   $M_N^+ + M_N^- - m_N^+ -m_N^->\mu$; (2) the size of $\KK^+$ and $\KK^-$ is larger than the size of any hole, i.e.\ the \emph{size condition} holds and in addition (3) both Lemma~\ref{lemgapconditionfuture} and Lemma~\ref{lemgapconditionpast} hold. 
According to Lemma~\ref{lemgapconditionfuture} the Cantor set $\KK^{+}$ admits a slow subdivision satisfying the gap condition.
Indeed, the holes in $\KK^{+}$ are described in Lemma~\ref{lemdescriptionholes} and the first $2d-1=|\cA|-1$ levels $\KK^{+}(1)\supset\dots\supset\KK^{+}(2d-1)$ of the subdivision are defined just removing, in any order, the holes corresponding to holes of first generation.
Similarly, the next $(2d-1)^2$ levels of the subdivision are defined removing, in any order, the holes of second generation and so on.
It is clear that at every step the intervals $K^L, K^R \subset \KK^+(n)$ as in property $(2)$ of the definition of slow subdivision  are larger than the ones considered in  Lemma~\ref{lemgapconditionfuture} and hence the gap condition holds.
Similarly, Lemma~\ref{lemdescriptionholes} and Lemma~\ref{lemgapconditionpast} imply that the Cantor set $\KK^{-}$ admits a slow subdivision satisfying the gap condition. 
Hence, applying Hall's Theorem  (Theorem~\ref{thm:Hall}) one gets the desired conclusion.  
\end{proofof}

\appendix

\section{Hall's Theorem on the sums of Cantor sets}\label{sec:thmHall}
In this appendix we include, for completeness' sake, the proof of Theorem~\ref{thm:Hall}.

Consider any Cantor set $\KK$ and let $(B_i)_{i\in\NN}$ be the collection of its holes.
A slow subdivision $(\KK(n))_{n\in\NN}$ for $\KK$ is called a \emph{monotone slow subdivision} if  the holes $B_1, \dots, B_n, \dots $ such that for any $n$ $\KK(n+1)=\KK(n)\setminus B_{n+1}$ are ordered by size, i.e.\ $|B_{n+1}|\leq |B_n|$ for any $n$.
It is hence clear that  a monotone slow subdivision always exist. 

Let $K$ be a compact interval and let $B$ be an open interval with $B\subset K$, where the inclusion is obviously strict.
Define the two closed subintervals $K^L$ and $K^R$ of $K$ such that $K=K^L\sqcup B\sqcup K^R$.  As suggested by the notation, we assume that $K^L$ is on the left side of $B$ and $K^R$ is on the right side of $B$. 

\begin{lem}\label{LemGapConditionForGeometricSubdivision}
Let $(\KK(n))_{n\in\NN}$ be a monotone slow subdivision for the Cantor set $\KK$. 
If $\KK$ admits another slow subdivision $(\widetilde{\KK}(n))_{n\in\NN}$ which satisfies the gap condition, then also the monotone slow subdivision $(\KK(n))_{n\in\NN}$ satisfies the gap condition. 
\end{lem}

\begin{proof}
Fix any $n$, consider the interval $K$ in the level $\KK(n)$ and  the hole $B_K \subset K$ such that $\KK(n+1) \cap K = K \setminus B_K$. Let  $K^L$ and $K^R$  be as usual 
 the closed subintervals of $K$ such that $K=K^L\sqcup B_K \sqcup K^R$.   Let us show that $|B_K|<|K^R|$, the proof of $|B_K|<|K^L|$ being the same. 
 
Consider the integer $m$ and the interval $\widetilde{K}$ of the $\widetilde{\KK}(m)$ such that $B_{\widetilde{K}} = B_K$, i.e.\  $\widetilde{\KK}(m+1)\cap\widetilde{K}=\widetilde{K}\setminus B_K$  and let $\widetilde{K}^L$ and $\widetilde{K}^R$ be the closed subintervals of $\widetilde{K}$ such that $\widetilde{K}=\widetilde{K}^L\sqcup B_K \sqcup \widetilde{K}^R$. 
Let $B$ be the hole in $\KK$ whose left endpoint coincides with the right endpoint of $K$ (and of $K^R$).
Remark that since the slow subdivision $((\KK(n))_{n\in\NN}$ is a monotone slow subdivision then we have $|B|\geq|B_K|$.  
Since  by assumption $(\widetilde{\KK}(n))_{n\in\NN}$ satisfies the gap condition, we have  $|B_K|<|\widetilde{K}^R|$. If by contradiction  $|B_K|\geq|K^R|$, since  $K^R$ and $\widetilde{K}^R$ have the same endpoints,   $K^R$ must be strictly contained in $\widetilde{K}^R$.  Equivalently, this means that we have $\widetilde{K}^R\cap B \neq\emptyset$.
 Remark that for any given slow subdivision for $\KK$, any interval of any level of the subdivision strictly contains all the holes which intersect it. Thus we must have $B\subset\widetilde{K}^R$.
 Hence, the interval $\widetilde{K}^R$  contains the hole $B$ and $K^R$ is a connected component of $\widetilde{K}^R\setminus B$ and  $|B|\geq|B_K| \geq|K^R|$. Since for some $m'>m$ the hole $B$ is removed from $\widetilde{K}^R$, this implies that the gap condition cannot hold for  $\widetilde{\KK}(m')$, thus giving  a contradiction and concluding the proof by absurd that $|B_K|<|K^R|$. 
\end{proof}

Hall's Theorem is proved iterating the simple argument stated in the Lemma below, whose proof is left to the reader (see also Lemma 2 in Chapter 4 in \cite{CF}).

\begin{lem}\label{LemSumOfIntervals}
Let $K$ and $F$ be two compact intervals.
Let $B$ be an open interval contained in $K$.
If $|B|<|F|$ then
\[
	K+F=(K^L+F)\cup (K^R+F).
\]
\end{lem}
Remark that the intervals $K^L+F$ and $K^R+F$ in the Lemma are not disjoint, indeed they are closed and the sum $K+F$ is connected. 

\begin{proofof}{Theorem}{thm:Hall}
Let $\left(\KK(n)\right)_{n\in\NN}$ and $\left(\FF(n)\right)_{n\in\NN}$ be slow monotone  subdivisions   respectively for $\KK$ and $\FF$. Since by assumption  $\KK$ and $\FF$ admit a slow subdivision which satisfy the gap condition, by Lemma~\ref{LemGapConditionForGeometricSubdivision} also  $\left(\KK(n)\right)_{n\in\NN}$ and $\left(\FF(n)\right)_{n\in\NN}$ satisfy the gap condition. 
  Set $K_0:=[\min\KK,\max\KK]$ and $F_0:=[\min\FF,\max\FF]$ and fix any point $x\in K_0+F_0$.
The Theorem follows if we show that we can construct two sequences $(n_i)_{i \in \NN}$ and  $(m_i)_{i \in \NN}$ such that $n_i \to \infty$, $m_i \to \infty$  and two sequences of nested closed intervals $(K_i)_{i\in\NN}$ and $(F_i)_{i\in\NN}$, where $K_i$ is an interval of the level $\KK(n_i)$ and $F_i$ is an interval of the level $\FF(m_i)$, such that $x\in K_i+F_i$ for any $i \in \NN$. 
Indeed setting $k:=\bigcap_{i\in\NN}K_i$ and $f:=\bigcap_{j\in\NN}F_j$ one has $x=k+f$ with $k\in\KK$ and $f\in\FF$.

We will construct the sequences  $(n_i)_{i \in \NN}$ and  $(m_i)_{i \in \NN}$ and the  two families of nested intervals by induction on  $i$  in $\NN$.  Fix $i$  in $\NN$ and assume that respectively the first $i+1$ nested intervals $K_0\supset K_1\supset\dots\supset K_i$ and the first $i+1$ nested intervals 
$F_0\supset F_1\supset\dots\supset F_i$ are defined.  Let $n(K_i)$ be the minimum $n \in \NN$ such that $K_i \cap \KK(n) \neq K_i $ and let $B_i$ be the hole in $K_i$, i.e.\ the  open subinterval $B_i\subset K_i$ such that $\KK(n(K_i))\cap K_i=K_i\setminus B_i$.
Similarly, let $n(F_i)$ be the minimum $n \in \NN$ such that $F_i \cap \FF(n) \neq F_i $ and let $C_i$ be the hole in $F_i$, i.e.\ the  open subinterval $C_i\subset F_i$ such that $\FF(n(F_i))\cap F_i=F_i\setminus C_i$. 

We will simultaneously prove by induction that for every $i$ the intervals $(K_i, F_i)$ and the holes $B_i \subset K_i$, $C_i \subset F_i$ in our construction satisfy the condition
\begin{equation}\label{EqThmHallTheorem}
	|B_i|<|F_i| \quad \text{ and } \quad |C_i|<|K_i|.
\end{equation}
Observe that for $i=0$ the condition is true by  the size condition which is assumed in the statement. To define the intervals at level $i+1$, we   subdivide the interval having the bigger hole. 
Assume that $|B_i|\geq|C_i|$, the other case being the same.
Since $|B_i|<|F_i|$ then Lemma~\ref{LemSumOfIntervals} implies $K_i+F_i=(K_i^L+F_i)\cup (K_i^R+F_i)$.  
Assume without loss of generality that $x\in K_i^L+F_i$ and set $K_{i+1}:=K_i^L$ and  $n_{i+1}=n(K_i)$, so that $K_{i+1} \in \KK(n_{i+1})$.
Set also $F_{i+1}=F_i$ and $m_{i+1}=m_i$, so $F_{i+1} \in \FF(m_{i+1})$ holds trivially. 
By the property of a monotone slow subdivision, the hole $B_{i+1}\subset K_{i+1}$ satisfies $|B_{i+1}|\leq|B_i|$ and therefore $|B_{i+1}|<|F_i|$ by inductive hypothesis. 
On the other hand the gap condition implies $|B_i|<|K_i^L|=|K_{i+1}|$ and therefore, since $C_i$ is by choice the smaller of the two holes,  $|C_i|\leq|B_i|<|K_{i+1}|$.
Thus, the pair of intervals $(K_{i+1},F_{i+1})$, with holes $B_{i+1}$ and $C_{i+1}=C_i$ satisfies the same condition~\eqref{EqThmHallTheorem} as the pair $(K_i,F_i)$ with holes $B_i$ and $C_i$, and moreover we have $x\in K_{i+1}+F_{i+1}$.
The inductive step is complete. Finally, since the holes of a Cantor set which are longer than a given $\epsilon>0$ are just finitely many, it is clear that both $(n_i)_{i \in \NN}$ and $(m_i)_{i \in \NN}$ are increasing sequences.  The Theorem is proved. 
\end{proofof}

\subsection*{Acknowledgements}
We would like to thank P.~Hubert, R.~Mukamel and C.~Series for useful discussions.
We would also like to thank the hospitality given by ICERM during the conference \emph{Geometric Structures in Low-Dimensional Dynamics}, by the Israel Institute for Advanced Studies in Jerusalem during the program \emph{Arithmetic \& Dynamics} and by the Max Planck Institute during the program \emph{Dynamics and Numbers},  where parts of this project were done.
Some of the research visits which made this collaboration possible were supported by the EPSRC Grant EP/I019030/1 and the ERC Grant ChaParDyn.
Ulcigrai is currently supported by the ERC Grant ChaParDyn.
Artigiani is supported by an EPSRC Doctoral Training Grant.



\end{document}